\documentclass[a4paper, 12pt]{article}
\usepackage{anysize}
\marginsize{3,5cm}{2,5cm}{2,5cm}{2,5cm}
\usepackage{amssymb}
\usepackage{amsmath,amsbsy,amssymb,amscd}
\usepackage{t1enc}
\usepackage[cp1250]{inputenc}
\usepackage[british]{babel}
\usepackage[all]{xy}
\usepackage{color}
\usepackage{amsfonts}
\usepackage{latexsym}
\usepackage{amsthm}
\usepackage{mathrsfs}
\usepackage{hyperref}
\usepackage{graphicx}

\usepackage{color}
\usepackage{caption}
\usepackage{subcaption}
\usepackage{enumerate}

\definecolor{orange}{rgb}{1,0.5,0}

\usepackage{mathrsfs} 

\DeclareMathAlphabet{\mathpzc}{OT1}{pzc}{L}{it} 

\def\a{\alpha}

\newtheorem{definition}{Definition}[section]

\newtheorem{proposition}[definition]{Proposition}
\newtheorem{theorem}{Theorem}

\newtheorem{corollary}[definition]{Corollary}
\newtheorem{remark}[definition]{Remark}
\newtheorem{lemma}[definition]{Lemma}

\def\C{\mathbb{C}}
\def\geq{\geqslant}
\def\leq{\leqslant}
\def\R{\mathbb{R}}
\def\T{\mathbb{T}}

\def\Z{\mathbb{Z}}
\def\N{\mathbb{N}}

\def\Q{\mathbb{Q}}
\def\cB{\mathcal{B}}

\def\cC{\mathcal{C}}

\newcommand{\bea}{\begin{eqnarray}}
  \newcommand{\eea}{\end{eqnarray}}
  \newcommand{\beab}{\begin{eqnarray*}}
  \newcommand{\eeab}{\end{eqnarray*}}
\renewcommand{\a}{\alpha}
  \newcommand{\be}{\begin{equation}}
  \newcommand{\ee}{\end{equation}}

\newcommand{\cD}{\mathcal D}

\title{Mutliple mixing and disjointness for time changes of bounded-type Heisenberg nilflows}
\author{Giovanni Forni and Adam Kanigowski}

\begin{document}
\baselineskip=14pt \maketitle
\begin{abstract} We study time changes of bounded type Heisenberg nilflows $(\phi_t)$ acting on the Heisenberg nilmanifold $M$. We show that for every positive $\tau\in W^s(M)$,  $s~>~7/2$, every non-trivial time change $(\phi_t^{\tau})$ enjoys the {\it Ratner property}. As a consequence every mixing time change is mixing of all orders. Moreover we show that for every  $\tau\in W^s(M)$, $s>9/2$ and every $p,q\in \N$, $p\neq q$,  $(\phi_{pt}^\tau)$ and $(\phi_{qt}^\tau)$ are disjoint. As a consequence {\it Sarnak Conjecture} on M\"obius disjointness holds for all such time changes. 
\end{abstract}

\tableofcontents

\section{Introduction}
In this paper we study ergodic properties of time changes of {\it Heisenberg nilflows}. Nilsystems on (non-abelian) nilmanifolds are classical examples of systems which share some features from both the {\it elliptic} and the {\it parabolic} world. They always have a non-trivial Kronecker factor which is responsible for the elliptic behavior (in particular they are never weakly mixing). On the other hand, orthogonally to the elliptic factor they  exhibit
polynomial speed of divergence of nearby trajectories and are polynomially mixing, which are properties characteristic of parabolic systems. 

We are interested in the lowest dimensional (non-abelian) situation, i.e. nilflows on $3$-dimensional Heisenberg nilmanifolds. In~\cite{AFU} it was shown that, for every (ergodic) Heisenberg nilflow, there exists a dense set of smooth time changes which are mixing. This result was strengthened in~\cite{FK}, where it was shown that for a full measure set of Heisenberg nilflows a {\it generic} time change is mixing, and moreover one has a ``stretched-polynomial'' decay of correlations for any pair of sufficiently smooth observables. For Heisenberg nilflows of {\it bounded type} the decay of correlations is estimated in~\cite{FK} to be polynomial, as expected according to the
``parabolic paradigm'' (see \cite{HK}, section 8.2.f).  The mixing result of~\cite{AFU} was generalized in~\cite{Rav} to a class of nilflows on higher step nilmanifolds, called quasi-Abelian, which includes suspension flows over toral skew-shifts, and then recently to all non-Abelian niflows in \cite{AFRU}. These general results reach no conclusion about the speed of mixing.

For time changes of horocycle flows, polynomial decay of correlations, as well as the Lebesgue spectral property,  were proved in \cite{FU}. For time changes of nilflows, even for Heisenberg nilflows of bounded type,  it is unclear whether the spectrum has an absolutely continuous component.

 It follows from \cite{AFU} and \cite{FK} that by a time change one can alter  the dynamical  features of Heisenberg niflows, i.e. the elliptic factor becomes trivial for the time-changed flow and the mixing property holds (with polynomial decay of correlations for bounded type nilflows). It is therefore natural to ask to what extent the time-changed flow can behave, roughly speaking, as a ``prototypical'' parabolic flow (there is no widely accepted formal definition of a parabolic system). One of the characteristic features of parabolic systems is the {\it Ratner property} which quantifies the polynomial speed of divergence of nearby trajectories. It was first established by 
 M.~Ratner in~\cite{Rat1} in the class of horocycle flows and was applied to prove Ratner's rigidity phenomena in this class. Moreover, in~\cite{Rat2}, M.~Ratner showed that the Ratner property survives under $C^1$ smooth time changes of horocycle flows, hence similar rigidity phenomena hold for time changes. One of the most important consequences of this property is that a mixing system with the Ratner property is mixing of all orders, see~\cite{RT}. 

Recently Ratner's property (or its variants) was observed in a new  class of (non-homogeneous) systems, that of smooth flows on surfaces with finitely many (saddle-like) singularities. In \cite{FayKa}, the authors studied the case of smooth mixing flows on the two-torus and established the SWR-property\footnote{Acronym for switchable weak Ratner. It was also shown that the original Ratner property does not hold.}. This property allows to establish the Ratner-type divergence of orbits, either in the future or in the past (depending on points), and moreover it has the same dynamical consequences as the original Ratner property. Then the authors showed in particular that the SWR-property holds for a full measure set of mixing flows with logarithmic singularities (Arnol'd flows) thereby proving higher order mixing in this class. The result in~\cite{FayKa} was strengthened in~\cite{KKU}, where the authors showed that the SWR-property holds for a full measure set of Arnol'd flows on surfaces of higher genus. 

It is therefore natural to ask whether a Ratner property holds in the class of Heisenberg nilflows. In \cite{KL} it is shown that Ratner's property implies in particular that the Kronecker factor is trivial and hence no nilflow can enjoy it. The situation is very different for {\it non-trivial time changes} of Heisenberg nilflows. 

\medskip 
Let $\mathfrak H$ denote the $3$-dimensional Heisenberg group and let $\mathfrak h$ denote its Lie algebra. Let $M:= \Gamma \backslash  \mathfrak H$ denote a Heisenberg 
nilmanifold, that is, the quotient of $\mathfrak H$ over a (co-compact) lattice  $\Gamma <\mathfrak H$.  For any $W\in \mathfrak h$, the flow $(\phi_t^{W})$ generated by the  projection to $M$ of the left-invariant vector field $W$ on $\mathfrak H$, is called a Heisenberg nilflow (see section \ref{sec.Heis} for the definition).

 A vector field $W\in \mathfrak h$ and the corresponding flow $(\phi_t^{W})$ are  called of {\it bounded type} if their projections on the Kronecker factor (which are respectively
 a constant coefficients vector field and the corresponding  linear flow on a $2$-dimensional torus) are of bounded type.

For any $W\in \mathfrak h$ and any positive function $\tau\in C^1(M)$, let $(\phi_t^{W,\tau})$ denote the time change of the nilflow $(\phi_t^{W})$, that is, the flow generated by the vector field $\tau W$ on $M$.  For every $s>0$, let $W^s(M)$ denote the standard Sobolev space. By the Sobolev embedding theorem we have that  $W^s(M) \subset C^k(M)$, for every $s> 3/2 + k$.

Our first main result establishes the Ratner property (see section \ref{sec:rat} for the definition)  for time changes of bounded type Heisenberg nilflows in a very strong sense. In fact our first main result is the following. 
\begin{theorem}\label{main:thm'} Let $W\in \mathfrak h$ be of bounded type. For any positive function $\tau\in W^s(M)$ with $s>7/2$,  either the time change is trivial ($1/\tau$ is
cohomologous to a constant for the nilflow $(\phi_t^W)$) or the time-changed flow $(\phi_t^{W,\tau})$ enjoys the Ratner property.
\end{theorem}

Recall that the famous Rokhlin problem asks whether mixing implies mixing of all orders.
The above result implies that the answer to the Rokhlin problem is positive for smooth time changes of bounded type Heisenberg nilflows:
\begin{corollary} Let $W\in \mathfrak h$ be of bounded type and let 
$$
\cD^s(W)=\{(\phi_t^{W,\tau})\;:\; \tau\in W^s(M), \tau>0\}.
$$
Then, for any $s\geq 7/2$, every element of $\cD^s(W)$ is mixing if and only if it is mixing of all orders.
 As proved in \cite{AFU} and \cite{FK}, for $s\geq 7/2$ and $W$ of bounded type, mixing is generic in the set $\cD^s(W)$.
\end{corollary}

Moreover, by \cite{KL}, we have the following strong dichotomy for time changes of Heisenberg nilflows:
\begin{corollary}\label{corA} Let $W\in \mathfrak h$ be of bounded type. Then for every positive function $\tau\in W^s(M)$ with $s>7/2$, either 
the time change is trivial ($1/\tau$ is cohomologous to a constant  for $(\phi_t^W)$) or $(\phi_t^{W,\tau})$ is mildly mixing (no non-trivial rigid factors). 
\end{corollary}

It turns out that  Heisenberg nilflows of bounded type (as in Theorem \ref{main:thm'}) are the only known examples, beyond horocycle flows and their time changes, for which the original Ratner property holds. 

Our second main result deals with disjointness properties of time changes of Heisenberg nilflows. It is based on a variant of a parabolic disjointness criterion from \cite{KLU}. We have:

\begin{theorem}\label{thm:main2'} Let $W\in \mathfrak h$ be of bounded type. For any positive function $\tau\in W^s(M)$ with $s>9/2$, if  the time change is non-trivial 
($1/\tau$ is not cohomologous to a constant  for $(\phi_t^W)$), then the flows $(\phi_{pt}^{W,\tau})$ and $(\phi_{qt}^{W,\tau})$ are disjoint for 
all $p,q\in \N$, $p\neq q$.
\end{theorem}

The above theorem should be compared with analogous disjointness results for other flows with Ratner's property. It follows from the renormalization equation for the horocycle flow,  which states that $g_sh_t=h_{e^{-2s}t}g_s$ for all $s,t\in \R$, that $h_{pt}$ and $h_{qt}$ are isomorphic (and hence not disjoint) for any $p,q\in \R\setminus\{0\}$. In \cite{Rat2}, joinings of time changes of horocycle flows were completely characterized by 
M.~Ratner. From Ratner's work one can derive that distinct powers of the same time change are disjoint unless the time change function is cohomologous to a constant \cite{FlaFo2}. A~different proof of this result, based on a new disjointness criterion for parabolic flows, was given recently in \cite{KLU}. Moreover in \cite{KLU} it is proved that for almost every Arnol'd flow on $\T^2$ the same assertion as in Theorem \ref{thm:main2'} holds. Therefore among known flows with Ratner's property, the horocycle flow is the only one for which  the conclusion of Theorem \ref{thm:main2'} does not hold. The heuristic reason for that is that the Ratner property for the horocycle flow depends only on the distance between points, and not on their position in space (since the space is homogeneous). In all other examples (for flows as in Theorem \ref{thm:main2'} in particular) the divergence depends also on position which allows to get stronger consequences (see section~\ref{sec:disj}).

Let us now briefly discuss the connection between Theorem \ref{thm:main2'} and Sarnak's Conjecture on M{\"o}bius disjointness \cite{Sar}, which is recently under extensive study, see e.g. \cite{FKL}. We say that a continuous flow $(T_t)$ on a compact metric space $(X,d)$ is M{\"o}bius disjoint, if for every $F\in C(X)$ and every $x\in X$ and every $t\in \R$ we have 
\begin{equation}\label{M}
\lim_{N\to+\infty}\frac{1}{N}\sum_{n\leq N}F(T_{nt}x)\mu(n)=0,
\end{equation}
here $\mu$ denotes the classical M{\"o}bius function\footnote{Sarnak's Conjecture states that every system of zero topological entropy is M{\"o}bius disjoint.}.

M{\"o}bius disjointness for horocycle flows was proved by J.~Bourgain, P.~Sarnak and T.~Ziegler~\cite{BSZ}. Moreover as explained in~\cite{FlaFo2}, it follows from Ratner's work~\cite{Rat2} that M{\"o}bius disjointness also holds for non-trivial time changes of horocycle flows. Moreover, in view of  a criterion due to Bourgain, Sarnak and Ziegler~\cite{BSZ}, for non-trivial time changes of horocycle flows the convergence in~\eqref{M} is uniform in $x\in X$. Uniform convergence is not known for horocycle flows. A corollary of Theorem \ref{thm:main2'}, again by the Bourgain-Sarnak-Ziegler criterion \cite{BSZ}, is the following:
\begin{corollary} Let $W\in \mathfrak h$ be of bounded type. For any positive function $\tau\in W^s(M)$ with $s>9/2$, if  the time change is non-trivial 
($1/\tau$ is not cohomologous to a constant  for $(\phi_t^W)$), then $(\phi_t^{W,\tau})$ is M{\"o}bius disjoint. Moreover the convergence in formula~\eqref{M} is uniform with respect to $x\in M$.
\end{corollary}
It follows from the work of B.~Green and T.~Tao~\cite{GT} that, if the time change is trivial  ($1/\tau$ is cohomologous to a constant   for $(\phi_t^W)$), then $(\phi_t^{W,\tau})$ is M{\"o}bius 
disjoint and that the convergence in formula~\eqref{M} is also uniform.

\medskip
The structure of the paper is as follows. In section \ref{sec.def} we recall some basic definitions of the theory of
joinings and recall the definition Heisenberg nilflows and their special flow representations over skew-shifts of the $2$-torus. In section~\ref{sec:rat} we recall the Ratner property and then formulate a version of it for  special flows.  In section~\ref{sec:disj} we state a disjointness criterion (Proposition \ref{disjoint.flows}) and then formulate a version of it for special flows (Lemma \ref{cocycle}). In section~\ref{sec:birksums} we derive from results of~\cite{FlaFo1} (see also \cite{FoSurvey}) estimates on Birkhoff sums for smooth functions over skew-shifts of the $2$-torus. Finally, in sections \ref{sec:proof} and \ref{sec:proof2} we prove our main theorems by applying the estimates from section \ref{sec:birksums}.

\section{Definitions}\label{sec.def}
\subsection{Joinings and disjointness}\label{sec.dis}

We refer the reader to \cite{Gla} for basic theory of joinings.  
Let $(\phi_t):(X,\cB,\mu)\to (X,\cB,\mu)$ and $(\psi_t):(Y,\cC,\nu)\to (X,\cB,\mu)$ be two ergodic flows. A {\it joining} of $(\phi_t)$ and $(\psi_t)$ is any $(\phi_t\times \psi_t)$ invariant measure such that $\rho(X\times B)=\mu(X)\nu(B)$ and $\rho(C\times Y)=\mu(C)\nu(Y)$. The set of joinings of $(\phi_t)$ and $(\psi_t)$ is denoted by $J((\phi_t),(\psi_t))$. Notice that $\mu\otimes \nu \in J((\phi_t),(\psi_t))$. We say that $(\phi_t)$ and $(\psi_t)$ are {\it disjoint} (denoting $(\phi_t)\perp (\psi_t)$) if $J((\phi_t),(\psi_t))=\{\mu\otimes \nu\}$.



\subsection{Heisenberg nilflows}
The (three-dimensional) Heisenberg group $\mathfrak H$ is given by 
$$
\mathfrak H:=\left\{\begin{pmatrix}1&x&z\\0&1&y\\0&0&1\end{pmatrix}\;:\; x,y,z\in \R\right\}.
$$
Let $\Gamma$ be a lattice in $\mathfrak H$. A {\em Heisenberg manifold} $M$ is a quotient $\Gamma \backslash \mathfrak H$. It is known that up to an automorphism of 
$\mathfrak H$
$$
\Gamma=\Gamma_K=\left\{\begin{pmatrix}1&m&\frac{p}{K}\\0&1&n\\0&0&1
\end{pmatrix}\;:\; m,n,p\in \Z\right\},
$$
where $K$ is a positive integer. Notice that $M$ has a (normalized) volume element $\text{ \rm vol}$  given by the projection of the (bi-invariant) Haar 
measure on $\mathfrak H$.  

Since the Abelianized Lie group $\bar{\mathfrak H}:=  \mathfrak H/[\mathfrak H,\mathfrak H]$ is isomorphic to $\R^2$ (as a Lie group) and $\bar \Gamma_K := \Gamma_K/[\Gamma_K,\Gamma_K]$ is a lattice in 
$\bar {\mathfrak H}$, the Heisenberg nilmanifold $M$ fibers over a  $2$-dimensional torus $\bar M= \bar {\mathfrak H}/ \bar \Gamma_K$, with fibers isomorphic to a circle. 

Let $W$ be any element of the Lie algebra $\mathfrak h$ of $\mathfrak H$. The {\em Heisenberg nilflow} for $W$ is given by
$$
\phi^W_t(x)=x\exp(tW)\,, \quad  \text{ for all } (x,t)\in M\times \R.
$$
Every Heisenberg nilflow $(\phi^W_t)$ on $M$ preserves the volume element $\rm vol$ on $M$.  
The classical ergodic theory of nilflows (see \cite{AGH}) implies that a Heisenberg nilflow $(\phi^W_t)$ is uniquely ergodic iff it is ergodic iff it is minimal iff the 
projected flow on $\bar M$ (which is isomorphic to its Kronecker factor) has rationally independent frequencies.  
More generally, the Diophantine properties  of a vector $W\in \mathfrak h$, and of the corresponding nilflow $(\phi^X_t)$,  under the renormalization dynamics
introduced in \cite{FlaFo1} can be entirely read from the Diophantine properties of the projection $\bar W$ of $W$ onto the Abelianized Lie algebra 
$\bar {\mathfrak h} := \mathfrak h/[\mathfrak h,\mathfrak h]$, which is also isomorphic to $\R^2$ (as a Lie algebra). In particular, a vector $W\in \mathfrak h$ is
called of {\it bounded type} if and only if is projection $\bar W \in \R^2$ is of bounded type.

\medskip
Let $\tau\in L^1(M), \tau>0$. The flow $(\phi_t^{W,\tau})$ is called a {\it time change} (or a reparametrization) of the flow $(\phi_t^W)$ if 
$$
\phi_t^{W,\tau}(x)=\phi^W_{\tau(x,t)}(x), \quad \text{ for all } (x,t) \in M\times \R\,,
$$
where the $(\phi^{W, \tau}_t)$-cocycle $\tau(x,t)$ is uniquely defined by the condition that 
$$
\int_{0}^{\tau(x,t)} 1/\tau(\phi_s^W(x))ds=t. 
$$

\subsection{Special flows}
Let $\Phi:(X, \cB, \mu,d)\to (X,\cB,\mu,d)$ be an ergodic automorphism of a compact metric probability space and let $f:X\to \R$ be strictly positive.

 We recall that the special flow $(\Phi_t):=(\Phi_t^{f})$ constructed above $\Phi$ and under $f$ acts on $X^f:=\{(x,s): x\in X, 0\leq s<f(x)\}$ by 
$$
\Phi_t(x,s)=(\Phi^{N(x,s,t)}(x),s+t-S_{N(x,s,t)}(f)(x)),
$$
where $N(x,s,t)$ is the unique integer such that 
\begin{equation}
\label{D-C}
0 \leq  s+t-   S_{N(x,s,t)}(f) (x) \leq f(\Phi^{N(x,s,t)}(x)),
\end{equation}    
and
$$
S_{N}(f) (x)=\left\{\begin{array}{ccc}
f(x)+\ldots+f(\Phi^{n-1}x) &\mbox{if} & n>0\\
0&\mbox{if}& n=0\\
-(f(\Phi^nx)+\ldots+f(\Phi^{-1}x))&\mbox{if} &n<0.\end{array}\right.$$
Notice that the flow $(\Phi_t)$ preserves the measure $\mu^f=\mu\otimes \lambda_\R$ restricted to $X^f$, where $\lambda_\R$ denotes the Lebesgue
measure on $\R$.

\subsection{Special flow representation of nilflows}\label{sec.Heis}
As shown in \cite{AFU} every {\it ergodic} nilflow $(\phi_t^W)$ can be represented as a {\it special flow}, where the base automorphism $\Phi_{\a,\beta}:\T^2\to \T^2$ is given by $\Phi_{\a,\beta}(x,y)=(x+\a,y+x+\beta)$ for $\a\in [0,1)\setminus \Q$, $\beta\in \R$ and under a constant roof function $f(x,y)=C_W>0$. 

Let $(q_n)_{n=1}^{+\infty}$ denote the sequence of denominators of $\a\in [0,1)\setminus \Q$. The vector field $W\in \mathfrak h$ is of {\it bounded type} if and only if $\a$ is of bounded type, i.e. there exists $C_\a>0$ such that $q_{n+1}\leq C_\a q_n$ for every $n\in \N$. 

For any function $f\in L^1(\T^2), f>0$, let $(\Phi_t^{f,\a, \beta})$ denote the special flow over $\Phi_{\a, \beta}$ and under $f$. Then every time change $\phi_t^{W,\tau}$ is isomorphic to a special flow $(\Phi_t^{f_{\tau},\a,\beta})$, where the roof function $f_\tau$ is as smooth as $\tau$. In view of the above representation, Theorems \ref{main:thm'} and \ref{thm:main2'} are respectively equivalent to the following two theorems:

\begin{theorem}\label{thm:main} Let $\a\in \R\setminus \Q$ be of bounded type and let $f\in W^s(\T^2)$, with $s>7/2$, be a positive function. Then the flow 
$(\Phi_t^{f,\a, \beta})$ has the Ratner property.
\end{theorem}

\begin{theorem}\label{thm:main2} Let $\a \in [0,1)\setminus \Q$ be of bounded type and let $f\in W^s(\T^2)$, with $s>9/2$, be a positive function. Then the flows $(\Phi_{pt}^{f,\a,\beta})$ and $(\Phi_{qt}^{f,\a,\beta})$ are disjoint for  all $p,q\in \N$ with $p\neq q$.
\end{theorem}

\begin{remark} It seems to the authors that a necessary condition for the Ratner property (or any of its variants) to hold in 
$$
\cD^s(W)=\{(\phi_t^{W,\tau})\;:\; \tau\in W^s(M), \tau>0\}.
$$
is that $W$ is of bounded type. 
\end{remark}

\section{Ratner's property}\label{sec:rat}
Let $(\phi_t):(X,\cB,\mu,d)\to (X,\cB,\mu,d)$ be an ergodic flow on a $\sigma$-compact metric probability space. 
\begin{definition}\label{def.rat} Let $P=\{-1,1\}$ and let $t_0\in \R$. The flow $(\phi_t)$ has  the {\it $R(t_0,P)$-property} if for every $\epsilon>0$ and $N\in \N$,  there exist 
$\kappa =\kappa(\epsilon)$, $\delta =\delta(\epsilon,N)$ and a set $Z = Z(\epsilon,N)$ with $\mu(Z)>1-\epsilon$,
such that:
for every $x,y\in Z$ with $d(x,y) <\delta$ and $x$ not in the orbit of $y$, there exist $p = p(x,y)\in P$ and $M = M(x,y), L = L(x,y)\geq N$, $\frac{L}{M}\geq \kappa$, for which
$$
\# \{n\in [M,M+L]\cap \Z\;:\; d(\phi_{nt_0}(x), \phi_{nt_0+p}(y))<\epsilon\}\geq (1-\epsilon)L.
$$
The flow $(\phi_t)$ is said to have {\it Ratner's property} if 
$$
\{s\in \R\;:\; (\phi_t)\text{ has property } R(s,P)\}
$$
is uncountable.
\end{definition}  

\subsection{Ratner's property for special flows}
In what follows $\Phi:(X,\cB,\mu,d)\to (X,\cB,\mu,d)$ is an ergodic automorphism of a $\sigma$-compact metric probability space and $f\in L_1^{+}(X)$. We have the following proposition (see Proposition 4.1. in~\cite{KK}):
\begin{proposition}\label{prop:specflows} Let $P=\{-1,1\}$. If for every $\epsilon>0$ and $N\in \N$ there exist $\kappa=\kappa(\epsilon)>0$, $\delta=\delta(\epsilon,N)>0$ and a set $Z=Z(\epsilon,N)\subset X$ with $\mu(Z)\geq 1-\epsilon$, such that, for every $x,y\in Z$ with  $d(x,y)\leq \delta$, there exist $M,L\geq N$, $\frac{L}{M}\geq \kappa$ and $p\in P$ such that for every $n\in [M,M+L]\cap \Z$
$$
d(\Phi^nx,\Phi^ny)< \epsilon\quad \text{ and }\quad |S_{n}(f) (x)-S_{n}(f)(y)-p|<\epsilon,
$$
then the special flow $(\Phi_t^f)$ satisfies the Ratner property.
\end{proposition}
In \cite{KK} Proposition~\ref{prop:specflows}  was proved for the SR-property, which is a modification of Ratner's property in which one allows for divergence either in the future or in the past. However the proof in \cite{KK} immediately extends  to a proof of Proposition \ref{prop:specflows}.

We will use Proposition \ref{prop:specflows} to prove Theorem \ref{thm:main}.

\section{Disjointness criterion}\label{sec:disj}
Let $(\phi_t):(X,\cB,\mu,d_1)\to (X,\cB,\mu,d_1)$ and $(\psi_t):(Y,\cC,\nu,d_2)\to (Y,\cC,\nu,d_2)$ be two weakly mixing flows (and $X,Y$ are $\sigma$-compact). In this section we prove the following proposition:

\begin{proposition}[Disjointness criterion]\label{disjoint.flows} 
Let $P'\subset \R$ be a compact set and let $v'\neq 0$. Fix $1\geq c>0$. Assume there exists $(A_k)\subset Aut(X_k,\cB|_{X_k},\mu|_{X_k})$, such that $\mu(X_k)\to \mu(X)$,
$A_k\to Id$ uniformly.
Assume moreover that for every $\epsilon>0$ and $N\in \N$  
there exist $(E_k=E_k(\epsilon))\subset \cB$ with $\mu(E_k)\geq c\mu(X)$, $0<\kappa=\kappa(\epsilon)<\epsilon$, $\delta=\delta(\epsilon,N)>0$, a set $Z=Z(\epsilon,N)\subset Y$ with  $\nu(Z)\geq (1-\epsilon)\nu(Y)$, such that for all $y,y'\in Z$ satisfying $d_2(y,y')<\delta$, every $k$ such that $d_1(A_k,Id)<\delta$ and every $x\in E_k$, $x':=A_{k}x$ there are $M\geq N$, $L\geq 1$, $\frac{L}{M}\geq \kappa$ and $V\in P'$, $v\in\{-v',v'\}$ ,
for which the following holds:
\begin{multline}\label{forw}
\max(d_1(\phi_tx,\phi_{t+V+v}x'),\;d_2(\psi_ty,\psi_{t+V}y'))<\epsilon\\ \text{ for } t\in U\subset [M,M+L]
\text{ with } \lambda_\R(U)\geq (1-\epsilon)L.
\end{multline}
Then $(\phi_t)$ and $(\psi_t)$ are disjoint.
\end{proposition}

The proof of the above proposition follows similar lines to the proof of Theorem~3 in \cite{KLU}.
We provide a proof for completeness.
\begin{proof}
Let $\rho\in J((\phi_t)_{t\in\R},(\psi_t)_{t\in \R})$ be an ergodic joining with $\rho\neq \mu\times \nu$. Since $(\phi_t)$ is weakly mixing it follows that, for $v\in\{-v',v'\}$, the map $\phi_v$ is ergodic and hence disjoint from $Id$. Therefore there exist $B_v\in \cB$ and $C_v\in \cC$ such that 

\begin{equation}\label{iddisT}|\rho(\phi_{-v}(B_v)\times C_v)-\rho(B_v\times C_v)|>\eta
\end{equation}
for some $0<\eta<1$. 
Let  $V^1_{\epsilon}(B_v):=\{x\in X\;:\; d_1(x,B_v)<\epsilon\}$ and similarly $V^2_{\epsilon}(C_v):=\{y\in Y\;:\; d_2(y,C_v)<\epsilon\}$.
There exists $0<\epsilon<\frac{c\eta}{1000}$ such that $$\max\left(\left|\mu(V^1_{\epsilon}(B_v))-\mu(B_v)\right|,
\left|\nu(V^2_{\epsilon}(C_v))-\nu(C_v)\right|\right) <\eta/32.$$ Since $\rho$ is a joining, by the triangle inequality, for each $t\in\R$, we have
\begin{equation}\label{pertu}
|\rho(\phi_{-t}V^1_{\epsilon}(B_v)\times V^2_\epsilon(C_v))-\rho(\phi_{-t}B_v\times C_v)|<\frac{\eta}{16}.
\end{equation}

By applying Birkhoff point-wise ergodic theorem to the joining flow $(\phi_t\times \psi_t, \rho)$ and to the characteristic functions of the sets 
$\phi_{-v}B_v\times C_v$ and $\phi_{-v}V^1_{\epsilon}(B_v)\times V^2_{\epsilon}(C_v)$
for $v\in \{-v',v'\}$, it follows that there exist $N_0\in \N$, $\kappa>0$ and a set $U_1\in \mathcal{B}\otimes \mathcal{C}$, with $\rho(U_1)>(1-\frac{c}{100})\rho(X\times Y)$,  
such that, for every $L,M\geq N_0$ with $\frac{L}{M}\geq \kappa$ and $v\in \{-v',v'\}\cup\{0\}$,  and for all  $(x,y)\in U_1$, we have
\begin{equation}\label{eeq1}\left|\frac{1}{L}\int_{M}^{M+L}\chi_{\phi_{-v}B_v
\times C_v}(\phi_tx,\psi_ty)\,dt-\rho(\phi_{-v}B_v\times C_v)\right|<\frac{\eta}{16},
\end{equation}
\begin{equation}\label{eeq3}\left|\frac{1}{L}\int_{M}^{M+L}
\chi_{\phi_{-v}V^1_{\epsilon}(B_v)
\times V^2_{\epsilon}(C_v)}(\phi_tx,\psi_ty)\,dt-\rho(\phi_{-v}V^1_{\epsilon}(B_v)\times V^2_\epsilon(C_v))\right|<\frac{\eta}{16}\,.
\end{equation}

Let $U_2:=U_1\cap (X\times Z)$, where $Z=Z(\epsilon,N_0)$ comes from our assumptions. Then $\rho(U_2)>(1-c/50)\rho(X\times Y)$.
 Note also that since $X\times Y$ is $\sigma$-compact, the measure $\rho$ is regular, and hence we can additionally assume that $U_2$ is compact.
Define ${\rm proj}:X\times Y\to X$, ${\rm proj}(x,y)=x$. Then the fibers of ${\rm proj}$ are $\sigma$-compact, and since $U_2$ is compact, the fibers of the map ${\rm proj}|_{U_2}:U_2\to {\rm proj}(U_2)\subset X$ are also $\sigma$-compact and ${\rm proj}(U_2)$ is also compact. Thus, by Kunugui's selection  theorem (see e.g.\ \cite{Ga-Le-Sch}, Thm. 4.1), it follows that there exists  a measurable (selection) $s_Y :{\rm proj}(U_2)\to X\times Y$ such that $(x,s_Y(x))\in U_2$. Note that $\mu({\rm proj}(U_2))\geq \rho(U_2)>(1-c/50)\mu(X)$. By Luzin's theorem there exists $X_{\rm cont}\subset {\rm proj}(U_2)$, with $\mu(X_{\rm cont})\geq (1-c/50)\mu(X)$, such that $s_Y$ is uniformly continuous on $X_{\rm cont}$. Finally, we set
$$
\widetilde{U}:=U_2\cap (X_{\rm cont}\times Y).
$$
We have $\rho(\widetilde{U})>(1-c/10)\rho(X\times Y)$. Moreover, for the set $U_X:={\rm proj}(\widetilde{U})$ we have
$\mu(U_X)\geq \rho(\widetilde{U})>(1-c/10)\rho(X\times Y)$.
Hence, by the definitions of sequences $(A_k)$ and $(E_k)=(E_k(\epsilon))$, it follows that there exists $k_0=k_0(\epsilon)$ such that for $k\geq k_0$, 
\begin{equation}\label{kurr}
\mu(A_{-k}(U_X\cap X_k)\cap (U_X\cap X_k)\cap E_k)>0.
\end{equation}
Let $\delta=\delta(\epsilon,N_0)$ come from the assumptions of our theorem. By the uniform continuity of $s_Y:X_{\rm cont}\to Y$ it follows that there exists $0<\delta'<\delta$ such that $d_1(x_1,x_2)<\delta'$ implies $d_2(s_Y(x_1),s_Y(x_2))<\delta$ for each $x_1,x_2\in X_{\rm cont}$.
Since $A_k\to Id$ uniformly
and $\widetilde{U}\subset X_{\rm cont}\times Y$, there exists $k_1=k_1(\epsilon)$ such that for $k\geq k_1$, $d_2(s_Y(x),s_Y(A_kx))<\delta$  for $x\in X_k\cap X_{\rm cont}$. Fix $k\geq \max(k_0,k_1+1)$ (so that $d_1(A_k,Id)<\delta'$). Let $x\in A_{-k}(U_X\cap X_k)\cap (U_X\cap X_k)\cap E_k$. Such a point does exist in view of~\eqref{kurr}. Set $x'=A_kx$, $y=s_Y(x)$, $y'=s_Y(x')$. By definition, $(x,y),(x',y')\in \widetilde{U}$ and $d_2(y,y')<\delta$ and all other  assumptions of our theorem are satisfied for $(x,y)$, $(x',y')$ (so that we obtain $M,L,V,v$ depending on $(x,y)$ and $(x',y')$ satisfying~\eqref{forw}).

We claim that
\begin{equation}\label{disjon}
\rho(\phi_{-v}(B_v)
\times C_v)>\rho(B_v\times C_v)-\frac{\eta}{2}.
\end{equation}
Indeed, in view of~\eqref{pertu}, the estimate~\eqref{disjon} follows if we can prove that
$$
\rho(\phi_{-v}(V^1_{\epsilon}(B_v))
\times V^2_{\epsilon}(C_v))>\rho(B_v\times C_v)-\frac{\eta}{4}.
$$

Using~\eqref{forw} (for $v=v'$),~\eqref{eeq1} (for $v=0$) and $(x,y)\in \widetilde{U}\subset U_1$, we obtain  that
\begin{multline}
\frac{1}{L}\int_{M}^{M+L}\chi_{V^1_{\epsilon}(B_v)
\times V^2_{\epsilon}(C_v)}(\phi_{t+V+v}x',\psi_{t+V}y')\,dt \\ > \frac{1}{L}\int_{M}^{M+L}\chi_{B_v
\times C_v}(\phi_tx,\psi_ty)\,dt-\epsilon >
\rho(B_v\times C_v)-\epsilon-\frac{\eta}{16}.
\end{multline}
Hence to complete the proof of claim~\eqref{disjon}, it is enough to show that
\begin{multline}\label{eq:suf}
\frac{1}{L}\int_{M}^{M+L}\chi_{V^1_{\epsilon}(B_v)
\times
V^2_\epsilon(C_v)}(\phi_{t+V+v}x,\psi_{t+V}y)\,dt<
\rho(\phi_{-v}V^1_{\epsilon}(B_v)
\times V^2_\epsilon(C_v))+\frac{\eta}{8}.
\end{multline}
Notice however that 
\begin{multline}
\frac{1}{L}\int_{M}^{M+L}\chi_{V^1_{\epsilon}(B_v)
\times
V^2_\epsilon(C_v)}(\phi_{t+V+v}x,\psi_{t+V}y)\,dt \\ =\frac{1}{L}\int_{M+V}^{M+V+L}\chi_{\phi_{-v}(V^1_{\epsilon}(B_v))
\times
V^2_\epsilon(C_v)}(\phi_{t}x,\psi_{t}y)\,dt\,,
\end{multline}
hence the estimate in~\eqref{eq:suf} follows from~\eqref{eeq3} with $M=M+V$ and $L=L$.

By a similar reasoning, we get
\begin{equation}\label{rhot}\rho(\phi_{-v}(B_v)\times C_v)<\rho(B_v\times C_v)+\frac{\eta}{2},
\end{equation}
so putting together~\eqref{disjon} and\eqref{rhot} we derive the estimate  $$|\rho(\phi_{-v}(B_v)\times C_v)-\rho(B_v\times C_v)|<\frac{\eta}{2}\,.$$ 
This however contradicts~\eqref{iddisT}, hence the argument is complete.
\end{proof}

\subsection{Disjointness criterion for special flows}
In this section we assume that $(\Phi_t)=(\Phi_t^f)$ and $(\Psi_t)=(\Psi_t^g)$ are special flows over  ergodic $\Phi\in Aut(X,\cB,\mu,d_1)$, $\Psi\in Aut(Y,\cC,\nu,d_2)$ respectively and under $f\in L^1_+(X,\cB,\mu)$, $g\in L^1_+(Y,\cC,\mu)$. Let $(\Phi_t^f)$ act on $X^f$ with metric $d_1^f$ and $(\Psi_t^g)$ act on $X^g$ with metric $d_2^g$. For $(x,s)\in X^f$ and $t\in \R$ we denote by $n(x,s,t)\in \Z$ the unique number for which

$$
S_{n(x,s,t)}(f)(x)\leq t+s<S_{n(x,s,t)+1}(f)(x),
$$
i.e.
\be\label{do1}
\Phi^f_{t}(x,s)=(\Phi^{n(x,s,t)}x,s+t-S_{n(x,s,t)}(f)(x)).\ee
We define $m(y,r,t)$  analogously for $(y,r)\in Y^g$.
We tacitly assume that
\be\label{do2}
\mbox{$f$ and $g$ are bounded away from zero.}
\ee
Before we state a disjointness criterion for special flows, we need the following general lemma:

\begin{lemma}\label{lem:bew} Fix $\frac{\inf_X f}{10}>\epsilon>0$. Let $t\in \R$ and $(x,s),(x',s')\in X^f$ be such that $|s-s'|<\epsilon^2$,  $d_1(\Phi^{n(x,s,t)}x,\Phi^{n(x,s,t)}x')<\epsilon^2$ and 
$$
\Phi_t^f(x,s)\in \{(x,s): \epsilon<s<f(x)-\epsilon\}.
$$
Let $V(t)=V(x,s,x',s', t):=S_{n(x,s,t)}(f)(x)-S_{n(x,s,t)}(f)(x')$. Then 
$$
d_1^f(\Phi_t^f(x,s),\Phi_{t-V(t)}^f(x',s'))\leq 2\epsilon^2.
$$
\end{lemma}
\begin{proof}Notice that since $\Phi_t^f(x,s)\in \{(x,s): \epsilon<s<f(x)-\epsilon\}$ and $|s-s'|<\epsilon^2$, we have
$$
S_{n(x,s,t)}(f)(x')\leq t-V(t)+s' \leq S_{n(x,s,t)+1}(f)(x').
$$
Therefore $\Phi_{t-V(t)}^f(x',s')=(\Phi^{n(x,s,t)}x',t-V(t)+s'-S_{n(x,s,t)}(f)(x'))$. By definition $\Phi_{t}^f(x,s)=(\Phi^{n(x,s,t)}x,t+s-S_{n(x,s,t)}(f)(x))$. The statement follows by the definition of $V(t)$ since $d_1^f$ is the product metric and we have $d_1(\Phi^{n(x,s,t)}x,\Phi^{n(x,s,t)}x')<\epsilon^2$ and $|s-s'|\leq\epsilon^2$. 
\end{proof}

\begin{lemma}\label{cocycle} Let $V\in \R$ and $P=\{-p',p'\}$ for $p'\neq 0$ and
$\zeta:=\frac{\int_Xf d\mu}{\int_Yg d\nu}$. Let $A_k\in Aut(X,\cB,\mu)$,
$A_k\to Id$ uniformly. Assume moreover that for every $\epsilon'>0$ and $N'\in \N$ there exist $0<\kappa'=\kappa'(\epsilon')<\epsilon'$, $\delta'=\delta'(\epsilon',N')>0$, such that for all $y,y'\in Y$ satisfying $d_2(y,y')<\delta'$, every $k$ such that $d_1(A_k,Id)<\delta'$ and every $x,x':=A_{k}x\in X$ there are $M'\geq N'$, $L\geq 1$, $\frac{L'}{M'}\geq \kappa'$ and $p\in P$ satisfying:
\begin{equation}\label{eq.bounded}
|S_{M'}(f)(x)-S_{M'}(f)(x')|<V
\end{equation}
\begin{equation}\label{forw.coc}
|(S_{M'}(f)(x)-S_{M'}(f)(x'))-(S_{[\zeta M']}(g)(y)-S_{[\zeta M']}(g)(y'))-p|<\epsilon'
\end{equation}
\begin{equation}\label{base:close}
d_1(\Phi^wx,\Phi^wx'), d_2(\Psi^wy,\Psi^wy')<\kappa'
\end{equation}
for $w\in [0,\max(1,\zeta)(M'+L')]\cap \Z$; and for $h=\{f,g\}$
\begin{equation}\label{forw.tight}
|(S_w(h)(x)-S_w(h)(x'))-(S_u(h)(x)-S_u(h)(x'))|<\epsilon',
\end{equation}
for every  $w,u\in[0,2\max(1,\zeta) (M'+L')]$, $|w-u|\leq L'$. 
Then $(\Phi_t^f)$ and $(\Psi_t^g)$ are disjoint.
\end{lemma}

The proof of the above proposition follows similar lines (although is simpler) than the proof of Proposition 4.1. in \cite{KLU}. We provide a proof here for completeness. 

\begin{proof}[Proof of Lemma \ref{cocycle}]
We will show that the assumptions of Proposition \ref{disjoint.flows} are satisfied. Let $P':=[-2V-|p'|,2V+|p'|]$ and $v'=p'$. Let $c=1$, and $A_k^f(x,s)=(A_kx,s)$ on $X^f$ (then $A_k^f\to Id$ uniformly). 
Fix $\epsilon>0$ and $N\in \N$. Let $\epsilon'=\epsilon^4$, N'= and let $\kappa', \delta'$ be as in Lemma \ref{cocycle}. Define $\kappa:=\kappa'^2$ and $\delta:= \min(\epsilon^{10},\delta'^2)$.

By ergodic theorem for $\Phi_t^f$ there exist $N''\in \R$ and a set $E=E(\epsilon)\subset X^f$, 
 $\mu^f(E)>1-\epsilon$, such that for every $M,L\geq N''$, $\frac{L}{M}\geq \kappa$ and every $(x,s)\in E$, we have 
\begin{equation}\label{eq:bet1}
\lambda_\R\left(\left\{t\in [M,M+L]\;:\; \Phi_t^f(x,s)\in \{(x,s):\epsilon^{3/2}<s<f(x)-\epsilon^{3/2}\}\right\}\right)\geq (1-\epsilon^{4/3})L 
\end{equation}
and for $t\in \R$ such that $n(x,s,t)\in [M,M+L]$, we have 
\begin{equation}\label{eq:nxt}
|t- n(x,s,t)\int_X fd\mu|<\kappa'^2 t.
\end{equation}
Similarly, by ergodic theorem for $\Psi_t^g$ there exist $N''\in \R$ and $Z=Z(\epsilon)\subset Y^f$,
$\nu^g(Z)>1-\epsilon$ such that for every $M,L\geq N''$, $\frac{L}{M}\geq \kappa$ and every $(y,r)\in Z$, we have 
\begin{equation}\label{eq:bet2}
\lambda_\R\left(\left\{t\in [M,M+L]\;:\; \Psi_t^g(y,r)\in \{(y,r):\epsilon^{3/2}<r<g(y)-\epsilon^{3/2}\}\right\}\right)\geq (1-\epsilon^{4/3})L,
\end{equation}
and for $t\in \R$ such that $m(y,r,t)\in [M,M+L]$, we have 
\begin{equation}\label{eq:nxt2}
|t- m(y,r,t)\int_Y gd\nu|<\kappa'^2 t.
\end{equation}

For $k\in \N$ let $E_k=E_k(\epsilon):=E$. Fix $E\ni (x,s), (x',s)=A_k(x,s)$ with $d_1^f((x,s),(x',s))<\delta$ and $(y,r), (y',r')\in Z$ with $d_2^g((y,r),(y',r'))<\delta$. Let M',L',p come from Lemma \ref{cocycle} for $x,x'$ and $y,y'$.  Define $M, L$ by $n(x,s,M)\int_X fd\mu =M'$ and $n(x,s,M+L)\int_X fd\mu=M'+L'$. It follows by \eqref{eq:nxt} that $\frac{L}{M}\geq \kappa$ and $M\geq N$. 

Let $U\subset [M,M+L]$ be such that for $t\in U$, we have $\Phi_t^f(x,s)\in \{(x,s):\epsilon^{3/2}<s<f(x)-\epsilon^{3/2}\}$, $\Psi_t^g(y,r)\in \{(y,r):\epsilon^{3/2}<r<g(y)-\epsilon^{3/2}\}$ and
\begin{equation}\label{sh3}
d_1(\Phi^{n(x,s,t)}x,\Phi^{n(x,s,t)}x'),
d_2(S^{m(y,r,t)}x,S^{m(y,r,t)}x')<\epsilon^2.
\end{equation}
By \eqref{eq:bet1}, \eqref{eq:bet2} and \eqref{base:close} it follows that $|U|\geq (1-\epsilon)L$. Let us then set 
$$
\begin{aligned}
V(t)&=S_{n(x,s,t)}(f)(x)-S_{n(x,s,t)}(f)(x')\,,\\  W(t)&=S_{m(y,r,t)}(g)(y)-S_{m(y,r,t)}(g)(y')\,.
\end{aligned}
$$
By Lemma \ref{lem:bew} (for $(\Phi_t^f)$ and $(\Psi_t^g)$) it follows that 
$$
d_1^f(\Phi_t^f(x,s),\Phi_{t-W(t)+(W(t)-V(t))}^f(x',s))\leq 2\epsilon^2.
$$
and
$$
d_2^g(\Psi_t^g(y,r),\Psi_{t-W(t))}^g(y',r'))\leq 2\epsilon^2.
$$

Notice that for $t\in [M,M+L]$, by \eqref{forw.tight} (and \eqref{eq:nxt},\eqref{eq:nxt2}), we have
\begin{equation}\label{sh1}
|V(t)-V(M)|,|W(t)-W(M))|\leq \epsilon^2
\end{equation}
 Moreover, by \eqref{forw.coc} and \eqref{forw.tight}, we have 
\begin{equation}\label{sh2}
|V(M)-W(M)-p|<\epsilon^2. 
\end{equation}
Hence 
$$
d_1^f(\Phi_t^f(x,s),\Phi_{t-W(M)+p)}^f(x',s))\leq 4\epsilon^2.
$$
and
$$
d_2^g(\Psi_t^g(y,r),\Psi_{t-W(M))}^g(y',r'))\leq 4\epsilon^2.
$$
Finally, by \eqref{eq.bounded} and \eqref{forw.tight} it follows that  $|W(M)|\leq 
|V(M)|+p+2\epsilon^2\leq 2V+p$ and hence $W(M)\in P$. This finishes the proof.
\end{proof}

We will use Proposition \ref{cocycle} to prove Theorem \ref{thm:main2}.

\section{Birkhoff sums over toral skew-shifts}\label{sec:birksums}
In what follows $\Phi_{\a,\beta}(x,y)=(x+\a, y+x+\beta)$ is a (linear) skew-shift  on $\T^2$, with $\a$ of bounded type and $g\in W^s(\T^2)$, with $s>7/2$ and 
$\int_{\T^2} g d \lambda_{\T^2} =0$, where $\lambda_{\T^2}$ denotes the normalized (Haar) Lebesgue measure on $\T^2$. We also assume that $g$ is not a coboundary (although some lemmas below are true also for coboundaries).  

\subsection{Cohomological equation for skew-shifts}

The cohomological equation for (linear) skew-shifts on $\T^2$ can be completely solved by Fourier series (see \cite{Katok:CC}, \cite{AFU})
Let $\Phi_{\a,\beta} :\T^2\to \T^2$ be given for $\a\in [0,1)\setminus \Q$ by the formula 
$$\Phi_{\a,\beta}(x,y)=(x+\a,y+x+\beta)\,.$$
It follows (see e.g. \cite{AFU}, \S 5) that
\begin{equation}\label{eq:dec}
L^2(\T^2)=\bigoplus_{(m,n)\in \Z_{|n|}\times\Z\setminus\{0\}}H_{m,n}
\end{equation}
where the spaces $H_{m,n}$ are $\Phi_{\a,\beta}$-invariant and 
$$
H_{m,n}=\bigoplus_{j\in\Z} \C e_{m+jn,n}\subset L^2(\T^2)
$$
with $e_{a,b}(x,y)=\exp(2\pi i(ax+by))$, for all $(a,b)\in \Z^2$.

\smallskip
The following result holds.

\begin{theorem} (\cite{Katok:CC}, Th. 11.25, \cite{AFU}, Th. 10) 
\label{thm:smoothcb}
For every $(m,n)\in \Z^2$, there exists a unique distributional obstruction to the existence of a smooth solution $u\in C^{\infty} (H_{m,n})$ of  the cohomological equation
$$
u\circ \Phi_{\a,\beta} - u =g 
$$
with right hand side $g \in C^{\infty} (H_{m,n})$. Such an obstruction is the invariant distribution $D_{m,n} \in W^{-s} (\T^2)$
for all $s>1/2$ defined as follows:
$$
D_{m,n}  ( e_{a, b} ) := 
\begin{cases} 
e^{-2\pi i [ (\a m +\beta n) j  +  \a n\binom{j}{2}] }  \quad &\text { \rm if }  (a, b) =(m+jn, n) \,; \\
   0                                      \quad  &\text{ \rm otherwise} .
\end{cases}
$$
The solution of the cohomological equation, for any function $g\in C^{\infty} (H_{m,n})$ such that
$D_{m,n} (g) =0$, is given by the following formula. If $g= \sum_{j\in \Z}  g_j e_{m+jn, n}$,
the solution $u= \sum_{j\in \Z}  u_j e_{m+jn, n}$ has Fourier coefficients:
\begin{equation}
\begin{aligned}
u_j &= - e^{2\pi i [(\a m +\beta n)j + \a n \binom{j}{2}]}  \sum_{k=-\infty}^j g_k 
e^{ -2\pi i[(\a m +\beta n) k + \a n \binom{k}{2}]} \\ &= 
 e^{2\pi i [(\a m +\beta n)j + \a n \binom{j}{2}]}  
 \sum_{k=j+1}^{\infty} g_k e^{ -2\pi i[(\a m +\beta n) k + \a n \binom{k}{2}]}  \,.
\end{aligned}
\end{equation}
If $g \in W^s (H_{m,n})$ for any $s>1$ and $D_{m,n} (g) =0$, then the above solution
$u\in W^t  (H_{m,n})$ for all $t<s-1$ and there exists a constant $C_{s,t} >0$ such that
$$
\Vert u\Vert _t \leq C_{s,t}  \, \Vert  g \Vert_s \,.
$$
\end{theorem} 

The results below establish the quantitative behavior of the square mean of ergodic averages for smooth 
functions under the skew-shift.

\begin{lemma}  
\label{lemma:L2bounds}
(\cite{AFU}, Lemma 15, or  \cite{FK}, Lemma 8.1)
Let  $(m,n)\in \Z_{\vert n\vert } \times \Z\setminus\{0\}$ and let $s>1/2$.
There exists a constant $C_s>0$ such that, for any  $g \in W^s(H_{m,n})$, 
\begin{equation}
\begin{aligned}
 C_s^{-1} \vert D_{m,n} (g) \vert  \leq   &\liminf_{N\to +\infty}    \frac{ 1} {N^{1/2}}  \Vert \sum_{k=0}^{N-1}  g \circ \Phi_{\a,\beta}^k 
\Vert_{L^2(\T^2)}  \\ 
&\leq \limsup_{N\to +\infty}    \frac{ 1} {N^{1/2}}  \Vert \sum_{k=0}^{N-1}  g \circ \Phi_{\a,\beta}^k 
\Vert_{L^2(\T^2)} \leq C_s \vert D_{(m,n)} (g) \vert\,.
\end{aligned}
\end{equation}
\end{lemma}

\subsection{General estimates}
\begin{lemma}\label{lem:upperbound} There exists a constant $C_{\a,g}>0$ such that for every $N\in \N$, and every $(x,y)\in \T^2$, we have 
$$
|S_{N}(g)(x,y)|\leq C_{\a,g}N^{1/2}.
$$
\end{lemma}
\begin{proof} Since $\a\in \R\setminus \Q$ is of bounded type, the statement follows from Lemma~1.4.9 of \cite{FoSurvey} or from Lemma 6.1. and Theorem 6.2. of \cite{FK}. 
In fact, since any constant roof suspension of $\Phi_{\a,\beta}$ is smoothly isomorphic to a Heisenberg nilflow $(\phi^W_t)$ on a nilmanifold $M$, generated by a bounded type vector field $W\in \mathfrak h$, for any $g\in W^s(\T^2)$ and every $(x,y)\in \T^2$, there exist a function $G\in W^s(M)$ and $p\in M$  such that
$$
S_{N}(g)(x,y)= \int_0^N  G\circ \phi^W_t(p) dt \,.
$$
By Lemma~1.4.9 of \cite{FoSurvey}  for any Heisenberg triple $\mathcal F:=(X,Y,Z)$ and for any $\sigma>2$, there exists a function $B_\sigma(\mathcal F,T)$ (defined in
formula $(1.71)$ of \cite{FoSurvey}) such that, for any function $f\in W^\sigma(M)$ and for all $(p,T)\in M\times \R$, 
$$
\vert  \frac{1}{T} \int_0^T  f\circ \phi^X_t(p) dt \vert  \leq \frac{ B_\sigma(\mathcal F,T)}{T}  \Vert f\Vert_\sigma \,.
$$
For $X=W$ of bounded type, by definition there exists a constant $C>0$ such that $B_\sigma(\mathcal F,T) \leq C_\sigma T^{1/2}$, hence we derive
$$
\vert S_{N}(g)(x,y) \vert \leq \vert \int_0^N  G\circ \phi^W_t(p) dt \vert  \leq C_\sigma N^{1/2} \Vert G\Vert_s\,
$$
(see also the comments after the proof of  Lemma~1.4.9 in \cite{FoSurvey}). Alternatively, from Theorem 6.2. of~\cite{FK} we derive that for 
$a=(X,Y,Z)$ satisfying an explicit Diophantine condition (depending only on $Y$) and for any $f\in W^s(M)$,  there exists a H\"older cocycle $\beta^f (a, p,T)$ such that
 $$
 \vert   \int_0^T  f\circ \phi^X_t(p) dt  -  \beta^f (a,p,T)\vert \leq C_s(a) \Vert f \Vert_s \,, 
 $$
and  Lemma 6.1 of \cite{FK} implies that whenever $X=W$ is of bounded type, the cocycle $\beta^f (a, p,T)$ satisfies the upper bound
$$
\vert \beta^f (a, p,T) \vert  \leq   C_s (a)  T^{1/2} \Vert f \Vert_s\,,
$$
which again implies our statement.
\end{proof}

\begin{lemma}\label{lem:gro} There exists a constant $c''>0$ such that for every $N\in \N$, 
$$
\|S_{N}(g)\|_{C^0(\T^2)}\geq c'' n^{1/2}.
$$
\end{lemma}
\begin{proof}
Since $g$ has zero average, but it is not a coboundary, and $\a\in \R\setminus \Q$ has bounded type,  it follows that we can assume that $g\in W^s(H_{m,n})$, for some 
$(m,n)$ with $n\neq 0$. In fact, otherwise $g$ is the pull back of function on the circle $\T$, which belongs to $W^s(\T)$ with $s> 7/2$, and since $\a$ is of constant
type, it follows by Fourier series that $g$ is a coboundary with transfer function $u \in W^t(\T)$ for all $t<s-1$ (in particular $u\in C^2(\T)$).

By orthogonality of the decomposition $W^s(\T^2)$ as a direct sum of components $W^s(H_{m,n})$, for all $s\in \R$, we can assume that $g\in W^s(H_{m,n})$, for some 
$(m,n)$ with  $n\neq 0$, hence  by  Lemma~\ref{lemma:L2bounds} there exists  $c'''>0$ such that, for all $N\in \N$, 
$$
\|S_{N}(g)\|_{L^2(\T^2)}\geq \|S_{N}(g_{m,n})\|_{L^2(\T^2)} \geq    c'''N^{1/2}.
$$
This finishes the proof.
\end{proof}

\begin{lemma}\label{lem:supr} For every $\chi>0$, $\zeta>1$ there exists $V_{\chi,\zeta}>1$ such that for every $T>0$ there exists $n_\chi,K_\chi\in [T,V_{\chi,\zeta} T]$ for which 
$$
\|S_{n_\chi}(g)\|_{C^0(\T^2)}\geq (1-\chi)\zeta^{-1/2}\|S_{[\zeta n_\chi]}(g)\|_{C^0(\T^2)}
$$
and, for $q=\zeta p $, 
$$
\|S_{qK_\chi}(g)\|_{C^0(\T^2)}\geq (1-\chi)\zeta^{1/2}\|S_{ pK_\chi}(g)\|_{C^0(\T^2)}.
$$
\end{lemma}
\begin{proof} For $\chi>0$ let $k\in \N$ be such that $(1-\chi)^k< \frac{c''}{C_{\a,g}}$ and let $V_{\chi,\zeta}:=\zeta^k$. 
By contradiction, if the statement is not true, then 
\begin{multline*}
c'' T^{1/2}\leq \|S_{T}(g)\|_{C^0(\T^2)}<(1-\chi)\zeta^{-1/2}\|S_{[\zeta T]}(g)\|_{C^0(\T^2)}\leq
\ldots\leq\\
 (1-\chi)^k\zeta^{-k/2}\|S_{[\zeta^k T]}(g)\|_{C^0(\T^2)}\leq C_{\a,g}(1-\chi)^kT^{1/2}.
\end{multline*}
This contradicts the choice of $k$. The proof of the second inequality follows the same lines.
\end{proof}

For $(a,b),(c,d)\in \T^2$ let $d_1((a,b),(c,d))=\|a-c\|$ and $d_2((a,b),(c,d))=\|b-d\|$. 

\begin{lemma}\label{lem:diver} There exists $C'=C'_{\a,g}>0$ such that, for every $(x,y),(x',y')\in \T^2$ and for every $n\in \N$, we have 
$$
|S_{n}(g)(x,y)-S_{n}(g)(x',y')|\leq C'\left(n^{3/2}d_1((x,y),(x',y'))+n^{1/2}d_2((x,y),(x',y'))\right)
$$
\end{lemma}
\begin{proof} By the mean value theorem for $S_{n}(g)$, we have for some $\theta_n\in \T^2$
\begin{equation}\label{eq:mvt}
|S_{n}(g)(x,y)-S_{n}(g)(x',y')|=\left|\frac{\partial S_{n}(g)}{\partial x}(\theta_n)(x-x')+
\frac{\partial S_{n}(g)}{\partial y}(\theta_n)(y-y')\right|.
\end{equation}
By the chain rule and Lemma \ref{lem:upperbound}
$$
|\frac{\partial S_{n}(g)}{\partial y}(\theta_n)|=|S_n(\frac{\partial g}{\partial y})(\theta_n)|\leq C_{\a,g}n^{1/2}. 
$$
Moreover, by the chain rule, we have 
$$
|\frac{\partial S_{n}(g)}{\partial x}(\theta_n)|\leq 
|S_n(\frac{\partial g}{\partial x})(\theta_n)|+
|\sum_{i=0}^{n-1}i\frac{\partial g}{\partial x}(\Phi_{\a,\beta} ^i\theta_n)|
$$
and by summation by parts
$$
|\sum_{i=0}^{n-1}i\frac{\partial g}{\partial x}(\Phi_{\a,\beta} ^i\theta_n)|\leq 
|nS_n(\frac{\partial g}{\partial x})(\theta_n)|+|\sum_{r=0}^{n-1}S_r(\frac{\partial g}{\partial x})(\theta_n)|.
$$
By Lemma \ref{lem:upperbound}, for some $C>0$,
$$
|nS_n(\frac{\partial g}{\partial x})(\theta_n)|<Cn^{3/2}\,\,\text{ and } \,\,
|\sum_{r=0}^{n-1}S_r(\frac{\partial g}{\partial x})(\theta_n)|<Cn^{3/2}.
$$
Using the above estimates in \eqref{eq:mvt} finishes the proof.
\end{proof}

\begin{lemma}\label{lem:lowbound}Fix $q\in \N$. For every $\eta>0$  there exists $D_\eta>1$ such that for every $n,m\in \N$ and every $(x,y)\in \T^2$, we have 
$$
\max_{i\in\{m,\dots,m+D_\eta n\}}|S_{n}(g)(\Phi_{\a,\beta}^{iq}(x,y))|\geq (1-\eta)\|S_{n}(g)\|_{C^0(\T^2)}.
$$
\end{lemma}
\begin{proof} Since $\a$ is of bounded type, there exists $D_\eta>1$ such that for every $(x,y)\in \T^2$ and every $n\in \N$, the orbit $\{\Phi_{\a,\beta}^{iq}(x,y)\}_{i=0}^{D_\eta n}$ is $(\frac{\eta^2}{n},\eta^2)$-dense in $\T^2$. Notice that if $(a,b)\in \T^2$ is any point such that 
$|S_n(g)(a,b)|=\|S_{n}(g)\|_{C^0(\T^2)}$ and $\|a-c\|\leq \frac{\eta}{n}$ and $\|b-d\|\leq \eta^2$, then by Lemma \ref{lem:diver}, 
$|S_n(g)(a,b)-S_n(g)(c,d)|\leq 2C'\eta^2 n^{1/2}$ which together with Lemma \ref{lem:gro} finishes the proof is $\eta>0$ is small enough.
\end{proof}

Recall that $(q_n)$ denotes the sequence of denominators of $\a$. The following simple lemma is a consequence of the pigeonhole principle:

\begin{lemma}\label{lem:zar} Fix $p,q\in \N$. For every $\eta>0$ there exists $L_\eta>0$ such that for every $(x,y),(z,w)\in \T^2$ and for every $n\in \N$, there exists $l_{1,n},l_{2,n}\in\{0,\ldots,L_\eta\}$, $l_{1,n}\neq l_{2,n}$ such that 
$$
d_2(\Phi_{\a,\beta}^{ql_{1,n}q_n}(x,y),(\Phi_{\a,\beta}^{ql_{2,n}q_n}(x,y)))<\eta \quad \text{and}\quad d_2(\Phi_{\a,\beta}^{pl_{1,n}q_n}(z,w),(\Phi_{\a,\beta}^{pl_{2,n}q_n}(z,w)))<\eta.
$$
\end{lemma}
\begin{proof} Let $L_\eta:=2\eta^{-2}$. By the pigeonhole principle there exist positive integers $l_1,\ldots,l_{[\eta^{-1}]+1}\in \{0,\ldots,L_\eta\}$ such that $d_2(\Phi_{\a,\beta}^{qlq_n}(x,y),(\Phi_{\a,\beta}^{ql'q_n}(x,y)))<\eta$ for $l,l'\in\{l_1,\ldots,l_{[\eta^{-1}]+1}\}$. Again by the pigeonhole principle, there exist  positive integers $l_{1,n},l_{2,n}\in \{l_1,\ldots,l_{[\eta^{-1}]+1}\}$ such that
$d_2(\Phi_{\a,\beta}^{pl_{1,n}q_n}(z,w),(\Phi_{\a,\beta}^{pl_{2,n}q_n}(z,w)))<\eta$. This finishes the proof.
\end{proof}

\begin{lemma}\label{lem:span} 
There exists a constant $C'''>0$ such that for any $(a,b),(c,d)\in \T^2$, for every $W,K\in \N$ with $W\leq \|a-c\|^{-1}$, we have 
\begin{multline}
|S_K(g)(\Phi_{\a,\beta}^W(a,b))-S_K(g)(\Phi_{\a,\beta}^{W}(c,d))-WS_K(\frac{\partial g}{\partial y})(\Phi_{\a,\beta}^W(a,b))(a-c)| \\
\leq C'''\left(K^{1/2}W^2\|a-c\|^2 + K^{3/2}\|a-c\|+K^{1/2}\|b-d\|\right).
\end{multline}
\end{lemma}
\begin{proof}Let $\theta_W=(a+W\a, d+Wc+\frac{W(W-1)\a}{2})$ and $\theta'_{W}=(a+W\a,b+Wc+\frac{W(W-1)\a}{2})$. 
Then 
\begin{multline*}
S_K(g)(\Phi_{\a,\beta}^W(a,b))-S_K(g)(\Phi_{\a,\beta}^{W}(c,d))=
\left(S_K(g)(\Phi_{\a,\beta}^W(a,b))-S_K(g)(\theta'_W)\right)+\\
\left(S_K(g)(\theta'_W)-S_K(g)(\theta_W)\right)+
\left(S_K(g)(\theta_W)-S_k(g)(\Phi_{\a,\beta}^W(c,d))\right).
\end{multline*}
By Lemma \ref{lem:diver}, we have 
$$
|S_K(g)(\theta'_W)-S_K(g)(\theta_W)|\leq C'K^{1/2}\|b-d\|
$$
and 
$$
|S_K(g)(\theta_W)-S_K(g)(\Phi_{\a,\beta}^W(c,d))|\leq C'K^{3/2}\|a-c\|.
$$
Finally, by Taylor formula and the chain rule, for some $\theta_K\in \T^2$,
\begin{multline}
S_K(g)(\Phi_{\a,\beta}^W(a,b))-S_K(g)(\theta'_W) \\=S_K(\frac{\partial g}{\partial y})(\Phi_{\a,\beta}^W(a,b))W(a-c)+S_K(\frac{\partial^2 g}{\partial^2 y})(\theta_K)(W(a-c))^2, 
\end{multline}
and by Lemma \ref{lem:upperbound}, we get 
$$
S_K(\frac{\partial^2 g}{\partial y^2})(\theta_K)(W(a-c))^2\leq C''
K^{1/2}W^2\|a-c\|^2. 
$$
This finishes the proof.
\end{proof}

\subsection{Estimates of second order terms}

In what follows $p,q\in \N$, $\zeta=\frac{q}{p}$, and a zero-mean non-coboundary $h\in W^s(M)$, $s\geq 5/2$ is fixed. We assume WLOG that $q>p$, that is $\zeta>1$.

The following lemma is important, it crucially uses the fact that $p\neq q$:
\begin{lemma}\label{lem:linsplit} There exist $D',d'>0$
such that, for any $(x,y), (z,w)\in \T^2$ and any $T>0$, for some $n'\in [0,D'T]$, we have
\begin{equation}\label{eq:linsplit}
|S_{n'}(h)(z,w)-\zeta^{-1/2}S_{[\zeta n']}(h)(x,y)|\geq d'T^{1/2}.
\end{equation}
\end{lemma}
\begin{proof} We will consider only numbers of the form $n'=pk$, for $k\in \N$. We will show that there exists $D',d'>0$ such that 
\begin{multline}\label{mu:show}
|\left(S_{pK}(h)(\Phi_{\a,\beta}^{pW}(z,w))-S_{pK}(h)(\Phi_{\a,\beta}^{pW}\Phi_{\a,\beta}^{pQ}(z,w))\right)\\
-\zeta^{-1/2}
\left(S_{qK}(h)(\Phi_{\a,\beta}^{qW}x,y)-S_{qK}(h)(\Phi_{\a,\beta}^{qW}\Phi_{\a,\beta}^{qQ}(x,y))\right)
|\geq 16d'T^{1/2},
\end{multline}
for some $K,W,Q\leq D' T$. This will finish the proof since by cocycle identity 

\begin{multline*}
S_{pK}(h)(\Phi_{\a,\beta}^{pW}(z,w))-S_{pK}(h)(\Phi_{\a,\beta}^{p(W+Q)}(z,w))=\\
S_{p(W+Q)}(h)(z,w)-S_{p(K+W+Q)}(h)(z,w)+S_{p(K+W)}(h)(z,w)-
S_{pW}(h)(z,w)
\end{multline*}
and the same splitting for $S_{qK}(h)(\cdot)$. Hence \eqref{eq:linsplit} then holds for $n'$ being one of $\{pW,p(K+W),p(W+Q),p(K+W+Q)\}$.

Let $\eta>0$ be small. By Lemma \ref{lem:supr}, let  $K\in [T,V_{\eta,\zeta}T]$ be such that 
\begin{equation}\label{eq2}
\left\|S_{qK}(\frac{\partial h}{\partial y})\right\|_{C^0(\T^2)}\geq (1-\eta)\zeta^{1/2}\left\|S_{pK}(\frac{\partial h}{\partial y})\right\|_{C^0(\T^2)}.
\end{equation}
Let now $n=n(K,\eta)>0$ be the smallest number such that 
$\frac{\eta^{3/2}q_{n+1}}{L_\eta}\geq  D_\eta qK$ ($L_\eta$ from Lemma \ref{lem:zar} and $D_\eta>1$ from Lemma \ref{lem:lowbound}). Let $l_{1,n},l_{2,n}\in \{0,\ldots,L_\eta\}$ be as in Lemma \ref{lem:zar} for $n,x,y,z,w$ and let $l_n=l_{2,n}-l_{1,n}$. Denote $(\bar{x},\bar{y})=\Phi_{\a,\beta}^{ql_{1,n}q_n}(x,y)$ and 
$(\bar{z},\bar{w})=\Phi_{\a,\beta}^{pl_{1,n}q_n}(z,w)$.

Notice that by definition,
$2\frac{\eta^{1/2}q_{n+1}}{l_n}-\frac{\eta^{1/2}q_{n+1}}{l_n}\geq D_\eta qK$, hence
 by Lemma \ref{lem:lowbound} there exists $W\in [\frac{\eta^{1/2}q_{n+1}}{l_n},2\frac{\eta^{1/2}q_{n+1}}{l_n}]$ such that 
\begin{equation}\label{eq1}
|S_{qK}(\frac{\partial h}{\partial y})(\Phi_{\a,\beta}^{qW}(\bar{x},\bar{y}))|\geq (1-\eta)\|S_{qK}(\frac{\partial h}{\partial y})\|_{C^0(\T^2) }\;\; (\geq c' q^{1/2}K^{1/2}).
\end{equation}
Therefore 
\begin{equation}\label{eq:000}
|S_{qK}(\frac{\partial h}{\partial y})(\Phi_{\a,\beta}^{qW}(\bar{x},\bar{y}))|
q^2Wl_n\|q_n\a\|\geq c'q^{5/2}K^{1/2}Wl_n\|q_n\a\|. 
\end{equation}

Since $Wl_n\|q_n\a\|\in [\eta^{1/2}/2, 2\eta^{1/2}]$, if $\eta\ll \min(p,q)$, we have for $b\in\{p,q\}$
$$
 \eta^{1/10} c'q^{5/2}K^{1/2}Wl_n\|q_n\a\|\geq (bK)^{1/2}(Wb^2l_n\|q_n\a\|)^2,
$$
 similarly 
$$
\eta^{1/10} c'q^{5/2}K^{1/2}Wl_n\|q_n\a\| \geq (bK)^{1/2}\eta.
$$
and since $W\geq \frac{\eta^{1/2}q_{n+1}}{2l_n}\geq \frac{\eta^{-1} D_\eta qK}{2}$, for $\eta>0$ sufficiently small, we have 
$$
\eta^{1/10} c'q^{5/2}K^{1/2}Wl_n\|q_n\a\|\geq  (bK)^{3/2}bl_n\|q_n\a\|.
$$

Notice that by Lemma~\ref{lem:span} for $(a,b)=(\bar{x},\bar{y})$ and $(c,d)=\Phi_{\a,\beta}^{ql_nq_n}(\bar{x},\bar{y})$, by formula~\eqref{eq:000} and the three above equations for $b=q$, we have 
\begin{multline*}
\vert S_{qK}(h)(\Phi_{\a,\beta}^{qW}(\bar{x},\bar{y})) - S_{qK}(h)(\Phi_{\a,\beta}^{qW}(\Phi_{\a,\beta}^{ql_nq_n}(\bar{x},\bar{y})))\vert \\ \geq
\vert S_{qK}(\frac{\partial h}{\partial y})(\Phi_{\a,\beta}^{qW}(\bar{x},\bar{y}))\vert   q^2W l_n \|q_n \a \| \\ 
- C''' \left( (q K)^{1/2}(Wq^2 l_n \Vert q_n \a \Vert )^2+(q K)^{1/2} \etaĘ + (q K)^{3/2} q l_n\Vert q_n \a \Vert \right)
 \\ \geq (1-3\eta^{1/10}C''')\left| S_{qK}(\frac{\partial h}{\partial y})(\Phi_{\a,\beta}^{qW}(\bar{x},\bar{y}))\right| q^2Wl_n\|q_n\a\|.
\end{multline*}
Moreover by \eqref{eq1} and \eqref{eq2}, we have
$$
|S_{qK}(\frac{\partial h}{\partial y})(\Phi_{\a,\beta}^{qW}(\bar{x},\bar{y}))|
\geq (1-\eta)^2\zeta^{1/2}
|S_{pK}(\frac{\partial h}{\partial y})(\Phi_{\a,\beta}^{pW}(\bar{z},\bar{w}))|.
$$

Therefore, by Lemma \ref{lem:span}, \eqref{eq:000} and the three above equations for $b=p$, we have
\begin{multline*}
|S_{pK}(h)(\Phi_{\a,\beta}^{pW}(\bar{z},\bar{w}))-
S_{pK}(h)(\Phi_{\a,\beta}^{pW}(\Phi_{\a,\beta}^{pl_nq_n}(\bar{z},\bar{w})))| \\  \leq
|S_{pK}(\frac{\partial h}{\partial y})(\Phi_{\a,\beta}^{pW}(\bar{z},\bar{w}))|p^2Wl_n\|q_n\a\| \\ +
C'''\left( (pK)^{1/2}(p^2Wl_n\|q_n\a\|)^2 +(pK)^{1/2}\eta   +
(pK)^{3/2}pl_n\|q_n\a\|\right) \\  \leq (p^2 (1-\eta)^{-2}\zeta^{-1/2} + 3\eta^{1/10}C''''  q^2)\left|S_{qK}(\frac{\partial h}{\partial y})
(\Phi_{\a,\beta}^{qW}(\bar{x},\bar{y}))\right|Wl_n\|q_n\a\|
\end{multline*}

Let $d_{p,q,\eta}:= q^2-  p^2 (1-\eta)^{-2}\zeta^{-1}  - C''''  (3+ \zeta^{-1/2})\eta^{1/10} q^2 $; then $d_{p,q,\eta}>0$ if  $\eta>0$ is small enough (since $q>p$). 

We have by the above
\begin{multline*}
|S_{qK}(h)(\Phi_{\a,\beta}^{qW}(\bar{x},\bar{y}))-
S_{qk}(h)(\Phi_{\a,\beta}^{qW}(\Phi_{\a,\beta}^{ql_nq_n}(\bar{x},\bar{y})))|\\
-\zeta^{-1/2}|S_{pK}(h)(\Phi_{\a,\beta}^{pW}(\bar{z},\bar{w}))-
S_{pK}(h)(\Phi_{\a,\beta}^{pW}(\Phi_{\a,\beta}^{pl_nq_n}(\bar{z},\bar{w})))|\geq \\
d_{p,q,\eta}\left|S_{qK}(\frac{\partial h}{\partial y})(\Phi_{\a,\beta}^{qW}(\bar{x},\bar{y}))\right|Wl_n\|q_n\a\|\geq\\
d'_{p,q,\eta}(qK)^{1/2}Wl_n\|q_n\a\|\geq d''_{p,q,\eta}K^{1/2}.
\end{multline*}
Since $K\geq T$ and by the definition of $(\bar{x},\bar{y}),(\bar{z},\bar{w})$, this finishes the proof of \eqref{mu:show} with  $K=K,W=W+l_{1,n}q_n$ and $Q=l_nq_n$,
hence the proof of Lemma~\ref{lem:linsplit} is complete.
\end{proof}

For $(x,y),(x,y')\in \T^2$, let $\delta_y=\vert y-y'\vert$.
We also have the following lemma:

\begin{lemma}\label{lem:vertsplit}
There exist $D''=D''(\a,f,p,q)>0$, $d''=d''(\a,f,p,q)>0$ such that for every
$(x,y),(x,y'),(z,w),(z,w')\in \T^2$, if $T':=\min(\delta_y^{-2}, \delta_w^{-2})\gg 1$, then there exists $s\in[0,D''T']$ such that 
\begin{equation}\label{eq:vertsplit}
|\left(S_{s}(p^{-1}f)(z,w)-S_{s}(p^{-1}f)(z,w')\right)-
\left(S_{[\zeta s]}(q^{-1}f)(x,y)-S_{[\zeta s]}(q^{-1}f)(x,y')\right)|
\geq d''.
\end{equation}
\end{lemma}
\begin{proof}
Notice that by Taylor formula and the chain rule, for some $\theta\in \T^2$
\begin{multline*}
S_s(p^{-1}f)(z,w)-S_s(p^{-1}f)(z,w') \\ =
(w-w')p^{-1}S_s(\frac{\partial f}{\partial y})(z,w) +
(w-w')^2p^{-1}S_s(\frac{\partial^2 f}{\partial^2 y})(\theta),
\end{multline*}
and by Lemma \ref{lem:upperbound} for $\frac{\partial^2 f}{\partial^2 y}$
$$
|(w-w')^2p^{-1}S_s(\frac{\partial^2 f}{\partial^2 y})(\theta)|\leq C_{\a,f} p^{-1}\delta_w^2s^{1/2}<T'^{-1/3}.
$$
for $s\leq T'^{11/10}$. An analogous reasoning for $S_{[\zeta s]}(q^{-1}f)(x,y)-S_{[\zeta s]}(q^{-1}f)(x,y')$   shows that \eqref{eq:vertsplit} follows by showing that there exist $D'',d''>0$ such that 
\begin{equation}\label{eq:ghj}
|(w-w')p^{-1}S_{s}(\frac{\partial f}{\partial y})(z,w)-(y-y')q^{-1}S_{[\zeta s]}(\frac{\partial f}{\partial y})(x,y)|\geq 2d''.
\end{equation}
for some $s\in [0, D''T']$.

To simplify notation let $h=\frac{\partial f}{\partial y}$. Let $\chi \in (0,1)$ be a small number, $\chi\ll \frac{d'}{D'}$  ($d',D'$ from Lemma \ref{lem:linsplit}).
We will consider two cases:\\
\textbf{A.}$\left||\frac{(w-w')q^{1/2}}{(y-y')p^{1/2}}|-1\right|\geq \chi$.

We will (WLOG) assume that $|(w-w')q^{1/2}|\geq (1+\chi)|(y-y')p^{1/2}|$, the proof in the case $|(w-w')q^{1/2}|\leq (1-\chi)|(y-y')p^{1/2}|$ is symmetric and follows the same lines. Notice that in this case 
$$ \quad T'\leq \frac{2p }{q }\delta_w^{-2}.$$

Let $V_{\chi/3,\zeta}>0$ come from Lemma \ref{lem:supr} and let $n_0\in [T', V_{\chi/3,\zeta}T']$ be as in the statement of Lemma \ref{lem:supr} for $g=h$. Then
\begin{equation}\label{eq:cor1}
\|S_{n_0}(h)\|_{C_0(\T^2)}\geq (1-\chi/3)\zeta^{-1/2}\|S_{[\zeta n_0]}(h)\|_{C_0(\T^2)}
\end{equation}

Let $D_{\chi/3}>0$ come from Lemma \ref{lem:gro} and let $u=u(n_0)\in \{0,\ldots,D_{\chi/3} n_0\}$ be such that 
\begin{equation}\label{eq:cor2}
|S_{n_0}(h)(\Phi_{\a,\beta}^u(z,w))|\geq (1-\frac{\chi}{3})\|S_{n_0}(h)\|_{C_0(\T^2)}.
\end{equation}
Let  $v=[\zeta (n_0+u)]-[\zeta n_0]$. We will show that 
\begin{equation}\label{eq:p1}
p^{-1}|S_{n_0}(h)(\Phi_{\a,\beta}^u(z,w))|
\vert w-w'\vert \geq (1+\chi/100)q^{-1} 
|S_{[\zeta n_0]}(h)(\Phi_{\a,\beta}^v(x,y))|\vert y-y'\vert
\end{equation}
Then 
\begin{multline*}
|(w-w')p^{-1}S_{n_0}(h)(\Phi_{\a,\beta}^u(z,w))-
(y-y')q^{-1}S_{[\zeta n_0]}(h)(\Phi_{\a,\beta}^v(x,y))|\geq \\
(1-(1+\chi/100)^{-1})\vert w-w'\vert p^{-1}|S_{n_0}(h)(\Phi_{\a,\beta}^u(z,w))|\geq d''\vert w-w'\vert n_0^{1/2}\\ \geq d''(\chi)\delta_w T'^{1/2}\geq d''(\chi) \,,
\end{multline*}
for some $d''(\chi)>0$. But then by cocycle identity we know that \eqref{eq:ghj} holds for $s=n_0$ or $s=n_0+u$ 
 and $d'(\chi)=d''(\chi)/2>0$.

Therefore it only remains to show \eqref{eq:p1}. By \eqref{eq:cor2} and \eqref{eq:cor1} and the assumptions of \textbf{A.}, we have
\begin{multline*}
p^{-1}|S_{n_0}(h)(\Phi_{\a,\beta}^u(z,w))|
\vert w-w'\vert \\ \geq p^{-1}(1-\frac{\chi}{3})^2\zeta^{-1/2}\|S_{[\zeta n_0]}(h)\|_{C_0(\T^2)}(1+\chi)\frac{p^{1/2}}{q^{1/2}}  \vert y-y'\vert \\ \geq q^{-1}(1+\chi/100)|S_{[\zeta n_0]}(h)(\Phi_{\a,\beta}^v(x,y))|
\vert y-y'\vert.
\end{multline*}
This finishes the proof of \eqref{eq:p1} and hence also the proof of 
case \textbf{A.}

\textbf{B.}$\left||\frac{(w-w')q^{1/2}}{(y-y')p^{1/2}}|-1\right|\leq \chi$.
In this case the LHS in \eqref{eq:ghj} is larger than
$$
\vert w-w' \vert p^{-1} \left|S_{s}(h)(z,w)-\zeta^{-1/2}S_{[\zeta s]}(h)(x,y)\right|-\chi\vert y-y'\vert q^{-1}|S_{[\zeta s]}(h)(x,y)|.$$
By Lemma \ref{lem:linsplit} for $T=T'$ and the definition of $T'$, there exists an $s_0\in [0,D'T']$ such that 
$$
\vert w-w'\vert p^{-1} \left|S_{s_0}(h)(z,w)-\zeta^{-1/2}S_{[\zeta s_0]}(h)(x,y)\right|\geq p^{-1} d'.
$$
Moreover, by Lemma \ref{lem:upperbound} and the definition of $T'$ for $h$ it follows that 
$$
\chi\vert y-y'\vert q^{-1}|S_{[\zeta s_0]}(h)(x,y)|\leq C_{\a,h}(D')^{1/2} \zeta^{1/2} q^{-1}\chi<p^{-1} \frac{d'}{10}, 
$$
if $\chi>0$ is sufficiently small. This finishes the proof of Lemma \ref{lem:vertsplit}.
\end{proof}


\section{Ratner's property: proof of Theorem \ref{thm:main}}\label{sec:proof}
In this section we will use the estimates from section \ref{sec:birksums} to prove Theorem \ref{thm:main}. We will use  Proposition \ref{prop:specflows}. Before we do that, we will prove a crucial proposition: 
For $(x,y),(x',y')\in \T^2$ denote $\vert x-x'\vert =\delta_x$ and $\vert y-y'\vert =\delta_y$.
\begin{proposition}\label{prop:split} For any $\a\in \R\setminus \Q$ of bounded type and for any  $f\in W^s(\T^2)$, $s>7/2$, there exists a constant $D_{\a,f}>0$ such that  the following holds. For every $(x,y),(x',y')\in \T^2$, if $T:=\min(\delta_x^{-2/3}, \delta_y^{-2})$, there exists $n_0=n_0((x,y),(x',y'))\in [0,D_{\a,f}T]\cap \Z$ such that 
\begin{equation}\label{eq:drift}
|S_{n_0}(f)(x,y)-S_{n_0}(f)(x',y')|>1.
\end{equation}
\end{proposition}
\begin{proof} Denote $a_n=S_{n}(f)(x,y)-S_{n}(f)(x',y')$, let $\eta<1/100$ and let $D_\eta>0$ come from Lemma \ref{lem:lowbound} for $g=\frac{\partial f}{\partial y}$. Let moreover $c''>0$ be as in Lemma \ref{lem:gro} for $g=\frac{\partial f}{\partial y}$ and $C'>c'' >0$ be as in Lemma \ref{lem:diver} for $g=f-\int_{\T^2} f d\lambda_{\T^2}$.
 Let $k:=[\frac{100T}{c''^2}]+1$. We will show that there exists $m\in [\frac{10C'k}{c''},\frac{10C'k}{c''}+D_\eta k+1]\cap \Z$ such that 
\begin{equation}\label{eq:cocyid}
|S_k(f)(\Phi_{\a,\beta}^m(x,y))-S_k(f)(\Phi_{\a,\beta}^m(x',y'))|\geq 3.
\end{equation}
This will finish the proof since then, by cocycle identity $|a_{m+k}-a_k|\geq 3$ and consequently \eqref{eq:drift} holds for $n_0=m+k$ or $n_0=k$.
By Lemma \ref{lem:lowbound} and \ref{lem:gro} there exists $m_1\in [\frac{10C'k}{c''},\frac{10C'k}{c''}+D_\eta k+1]\cap \Z$ and $m_2\in [10m_1,10m_1+D_\eta k+1]$ such 
that  
\begin{equation}\label{eq:mlarge}
|S_k(\frac{\partial f}{\partial y})(\Phi_{\a,\beta}^m(x,y))|\geq (1-\eta)c'' k^{1/2} \,,  \quad \text {\rm for } m=m_1, m_2\,.
\end{equation}
Moreover, since  $m_2\geq 10m_1$,  by the triangle inequality we have
$$
\max_{m\in\{m_1,m_2\}}\vert m(x-x')+(y-y')\vert \geq \max\left(\frac{m_1}{2}\delta_x,\frac{\delta_y}{2}\right).
$$
Let now $m$ denote the element above which attains the maximum. Since 
$m\geq \frac{10C'k}{c''}$, we have 
\begin{equation}\label{eq:qu}
(1-\eta)c''\vert m(x-x')+(y-y')\vert k^{1/2}\geq 2C'\delta_x k^{3/2}.
\end{equation}
Let $z_m:=( x+m\a, y'+m x'+\frac{m(m-1)\a}{2})$.  Notice that  
\begin{multline}\label{mult:split}
S_k(f)(\Phi_{\a,\beta}^m(x,y))-S_k(f)(\Phi_{\a,\beta}^m(x',y'))=
\left(S_k(f)(\Phi_{\a,\beta}^m(x,y))-S_k(f)(z_m)\right)+\\
\left(S_k(f)(z_m)-S_k(f)(\Phi_{\a,\beta}^m(x',y'))\right).
\end{multline}
By Lemma \ref{lem:diver} for $g=f-\int_{\T^2} f d\lambda_{\T^2}$, we have 
$$
|S_k(f)(\Phi_{\a,\beta}^m(x',y'))-S_k(f)(z_m)|\leq C' \delta_x k^{3/2}.
$$
By Taylor formula (and chain rule), for some $\theta_k\in \T^2$,
\begin{multline*}
S_k(f)(\Phi_{\a,\beta}^m(x,y))-S_k(f)(z_m)=(m(x-x')+(y-y'))S_k(\frac{\partial f}{\partial y})(\Phi_{\a,\beta}^m(x,y))+\\
(m(x-x')+(y-y'))^2
S_k(\frac{\partial^2 f}{\partial y^2})(\theta_k).
\end{multline*}
Since $m\leq C''' T$ (for some $C'''>0$), we have $\vert m(x-x')+(y-y')\vert \leq m\delta_x  + \delta_y \leq T^{-1/3}$, for $T>1$ large enough, and therefore by Lemma \ref{lem:upperbound} for $g=\frac{\partial f}{\partial y}$, we have
$$
|(m(x-x')+(y-y'))^2
S_k(\frac{\partial^2 f}{\partial y^2})(\theta_k)|\leq 1/100.
$$
Therefore, by \eqref{eq:qu}, \eqref{eq:mlarge} and \eqref{mult:split}, we have
\begin{multline*}
|S_{k}(f)(\Phi_{\a,\beta}^m(x,y))-S_{k}(f)(\Phi_{\a,\beta}^m(x',y'))| \\ \geq (1-\eta)c''\vert m\delta_x+\delta_y\vert k^{1/2}-1/100-C'\delta_x k^{3/2}\\ \geq 
\frac{c''}{3}\vert m\delta_x+\delta_y\vert k^{1/2}-1/100.
\end{multline*}
However, since $m_1\geq \frac{10C'k}{c''} >10k$,  we have 
$$ 
\frac{c''}{3}\vert m\delta_x+\delta_y\vert k^{1/2}\geq \frac{c''}{3}\max(\frac{m_1}{2}\delta_x,\frac{\delta_y}{2})k^{1/2} \geq 2,
$$
the last inequality since $k=[\frac{100T}{c''^2}]+1$. This finishes the proof of \eqref{eq:cocyid} and hence also the proof of Proposition \ref{prop:split}.
\end{proof}

Now we can prove Theorem \ref{thm:main}:

\begin{proof}[Proof of Theorem \ref{thm:main}]
We will use Proposition \ref{prop:specflows}.
Fix $\epsilon>0$ and $N\in \N$. Let  $Z=\T^2$ (see Definition \ref{def.rat}), and take any $(x,y),(x',y')\in \T^2$ with 
$d((x,y),(x',y'))=\delta_x+\delta_y<\delta$ (we will specify $\delta$ and $\kappa$ in the proof). Let 
\begin{equation}\label{def:time}
T:=\min(\delta_x^{-2/3},\delta_y^2).
\end{equation}

Denote $a_k=S_k(f)(x,y) -  S_k(f)(x',y')$ and let $D_{\a,f}>0$ be as in Proposition \ref{prop:split}. Notice that  
 for every $k\in [0,2D_{\a,f}T]$, we have
\begin{multline}\label{eq:distn}d_1((\Phi_{\a,\beta}^k(x,y)),(\Phi_{\a,\beta}^k(x',y')))=\vert x-x'\vert=
\delta_x\text{ and }\\
d_2((\Phi_{\a,\beta}^k(x,y)),(\Phi_{\a,\beta}^k(x',y')))\leq k\delta_x+\delta_y\leq 2D_{\a,f}\delta_x^{1/3}+\delta_y\,,
\end{multline}
hence by Lemma \ref{lem:diver} and by the cocycle identity, for any $\epsilon>0$ there exists $\delta_\epsilon>0$  such
that for $\delta\in (0, \delta_\epsilon)$, we have that $|a_{k+1} -a_k| < \epsilon/4$ for every $k \in  [0,D_{\a,f} T]$.
 Since by Proposition~\ref{prop:split}, there exists 
$n_0 \in [0,D_{\a,f}T]$ such that $\vert a_{n_0} \vert >1$, it follows that there exists $M\in [0,D_{\a,f}T]$ such that 

$$
\min \{ |a_M + 1|,  | a_M - 1|\} \leq \epsilon/3.
$$
Notice that by Lemma~\ref{lem:diver}, for every $N\in \N$ there exists $\delta_N>0$ such that for $\delta\in (0,\delta_N)$,  we have $M\geq N$.
Let $L=\kappa M$. Notice that $[M,M+L]\subset [0,2D_{\a,f}T]$ and therefore 
\begin{equation}\label{eq:eps}
d(\Phi_{\a,\beta}^n(x,y),\Phi_{\a,\beta}^n(x',y'))<\epsilon, \text{ for every } n\in [M,M+L]\,.
\end{equation}

Moreover, by the cocycle identity, for every  $n\in [M,M+L]$, we have  
$$
|a_n-a_M|=|S_{n-M}(f)(\Phi_{\a,\beta}^M(x,y))-S_{n-M}(f)(\Phi_{\a,\beta}^M(x',y'))|.
$$
By Lemma \ref{lem:diver} (for $g=f-\int_{\T^2}f d\lambda_{\T^2}$),  \eqref{eq:distn} for $k=M$ and \eqref{def:time}, we have (since $n-M\leq L=\kappa M\leq D_{\a,f}\kappa T$),
there exists $\kappa_\epsilon>0$ such that for  $\kappa\in (0, \kappa_\epsilon)$, we have
\begin{multline*}
|S_{n-M}(f)(\Phi_{\a,\beta}^M(x,y))-S_{n-M}(f)(\Phi_{\a,\beta}^M(x',y'))|\leq\\ C'(D_{\a,f}\kappa T)^{3/2}\delta_x+ C'(D_{\a,f}\kappa T)^{1/2}(2D_{\a,f}\delta_x^{1/3}+\delta_y)\leq\\
 C'(D_{\a,f}\kappa)^{3/2}+C'(D_{\a,f}\kappa)^{1/2}(2D_{\a,f}+1)\leq \epsilon/3.
\end{multline*}
Therefore, for every $n\in [M,M+L]$, 
$$
|a_n +  1|\leq |a_n-a_M|+|a_M+  1|\leq \epsilon \quad \text{ or } \quad |a_n -  1|\leq |a_n-a_M|+|a_M-  1| \leq \epsilon\,.
$$
This property, together with \eqref{eq:eps}, finishes the proof of  the hypoheses of Proposition~\ref{prop:specflows} and hence also the proof of Theorem~\ref{thm:main}.
\end{proof}

\section{Disjointness: proof of Theorem \ref{thm:main2}}\label{sec:proof2}
Notice that if $(\Phi_t^f)$ is a special flow over $T$ and under $f$ then $(\Phi_{rt}^f)$ is isomorphic to $(\Phi_t^{r^{-1}f})$. We will therefore use Proposition~\ref{cocycle} for 
$(\Phi_t^{\a,p^{-1}f})$ and $(\Phi_t^{\a,q^{-1}f})$ (notice that by Corollary \ref{corA} we know that $(\Phi_t^{\a,f})$, and hence also $(\Phi_t^{\a,p^{-1}f})$ and $(\Phi_t^{\a,q^{-1}f})$, are weakly mixing).

 We assume WLOG that $\zeta=\frac{q}{p}>1$.

Recall that for $(x,y),(x',y')\in \T^2$, $\delta_x=\vert x-x'\vert$ and $\delta_y=\vert y-y'\vert$. For the proof of Theorem \ref{thm:main2} we need the following crucial proposition:

\begin{proposition}\label{prop:cruc} There exists $D'''=D'''(\a,f,p,q)>0$ and $d'''=d'''(\a,f,p,q)>0$ such that for every $(x,y),(x',y'),(z,w),(z,w')\in \T^2$ if 
$$
T:=\min(\delta_{x}^{-2/3},\delta_{y}^{-2},\delta_{w}^{-2}),
$$
then for some $s\in [0,D'''T]$, we have
\begin{equation}\label{eq:nwc}
|\left(S_s(p^{-1}f)(z,w)-S_s(p^{-1}f)(z,w')\right)-
\left(S_{[\zeta s]}(q^{-1}f)(x,y)-S_{[\zeta s]}(q^{-1}f)(x',y')\right)|\geq 
d'''.
\end{equation}
\end{proposition}
\begin{proof} Let $C_{p,q}>1$ be a large constant (specified at the end of \textbf{Case 1}).
We will consider the following cases:

\textbf{Case 1.} $\delta_w\geq C_{p,q} \max(\delta_y,\delta_x^{1/3})$ or $\delta_w\leq C_{p,q}^{-1}\max(\delta_y,\delta_x^{1/3})$. 

If $\delta_w\geq C_{p,q}\max(\delta_y,\delta_x^{1/3})$, then $T=\delta_w^{-2}$. By Proposition \ref{prop:split} for $(z,w)$ and $(z,w')$ we have 
$$|S_s(p^{-1}f)(z,w)-S_s(p^{-1}f)(z,w')|\geq p^{-1}$$
for some $s\leq D_{\a,f}T$. Moreover, by Lemma \ref{lem:diver}, for $(x,y)$ and $(x',y')$ we have 
\begin{multline*}
|S_{[\zeta s]}(q^{-1}f)(x,y)-S_{[\zeta s]}(q^{-1}f)(x',y')| \\ \leq C'q^{-1}((\zeta D_{\a,f}T)^{3/2}\delta_x+(\zeta D_{\a,f}T)^{1/2}\delta_y)\leq 
\frac{C'q^{-1}\zeta^{3/2}D_{\a,f}^{3/2}}{C_{p,q}}\leq \frac{p^{-1}}{2} \,, 
\end{multline*}
if $C_{p,q}>1$ is large enough.

Similarly, if $\delta_w\leq C_{p,q}^{-1}\max(\delta_y,\delta_x^{1/3})$, using Proposition \ref{prop:split} for $(x,y)$ and $(x',y')$ we have 
$$|S_{[\zeta s]}(q^{-1}f)(x,y)-S_{[\zeta s]}(q^{-1}f)(x',y')|\geq q^{-1},$$
for some $s\leq \zeta^{-1}D_{\a,f}T$. Moreover, by Lemma \ref{lem:diver}, for $(z,w)$ and $(z,w')$ we have 
$$
|S_s(p^{-1}f)(z,w)-S_s(p^{-1}f)(z,w')|\leq p^{-1}\frac{C' D_{\a,f}^{1/2}\zeta^{-1/2}}{C_{p,q}}\leq \frac{q^{-1}}{2},
$$
if $C_{p,q}>1$ is large enough.
This finishes the proof of \textbf{Case 1.}\\

\textbf{Case 2.} $C_{p,q}\max(\delta_y,\delta_x^{1/3})\geq\delta_w\geq C_{p,q}^{-1}\max(\delta_y,\delta_x^{1/3})$. Let $R'=R'_{p,q}>1$ be a constant to be specified later (at the end of the proof of \textbf{Subcase 1}). 

We consider two subcases: \\
\textbf{Subcase 1.} $R'\delta_x^{1/3}\leq \delta_y$.
Notice that, by Lemma \ref{lem:diver}, we have, for all $k\in \N$,
$$
|S_{[\zeta k]}(q^{-1}f)(x',y')-
S_{[\zeta k]}(q^{-1}f)(x,y')|\leq q^{-1}C'(\zeta k)^{3/2}\delta_x\leq \frac{q^{-1}C'(\zeta k)^{3/2}}{R'^{3} T^{3/2}}.
$$
Moreover,  in this case we have $T =\min ( \delta_y^{-2}, \delta_w^{-2})$, hence by Lemma \ref{lem:vertsplit}, there exists $s\in [0,D''T]$ 
such that 
$$
|\left(S_s(p^{-1}f)(z,w)-S_s(p^{-1}f)(z,w')\right)-
\left(S_{[\zeta s]}(q^{-1}f)(x,y)-S_{[\zeta s]}(q^{-1}f)(x,y')\right)|\geq 
d''.
$$
Hence if only 
$\frac{q^{-1}C'(\zeta D''T)^{3/2}}{R'^{3}T^{3/2}}\leq d''/2$ (which is true if $R'$ is large enough), then \eqref{eq:nwc} follows by triangle inequality.

\textbf{Subcase 2.} $R'\delta_x^{1/3}\geq \delta_y$. 

Notice that by Lemma \ref{lem:diver}, for every $k, m\in \N$
\begin{equation}\label{eq:02}
|S_k(p^{-1}f)(\Phi_{\a,\beta}^m(z,w))-S_k(p^{-1}f)(\Phi_{\a,\beta}^m(z,w'))|\leq C'p^{-1}k^{1/2}\delta_w.
\end{equation}
Moreover, for every $m\in \N$ denote $h_m=(x'+m\a, y+mx+\frac{m(m-1)\a}{2})$. Then

\begin{multline}\label{sar:split2}
S_{[\zeta k]}(q^{-1}f)(\Phi_{\a,\beta}^m(x,y))-S_{[\zeta k]}(q^{-1}f)
(\Phi_{\a,\beta}^m(x',y'))\\ =
\left(S_{[\zeta k]}(q^{-1}f)(\Phi_{\a,\beta}^m(x,y))-S_{[\zeta k]}(q^{-1}f)(h_m)\right)\\ +
\left(S_{[\zeta k]}(q^{-1}f)(h_m)-S_{[\zeta k]}(q^{-1}f)(\Phi_{\a,\beta}^m(x',y'))\right).
\end{multline}
By Lemma \ref{lem:diver} for $g=f-\int_{\T^2} f d\lambda_{\T^2}$, we have 
\begin{equation}\label{eq:01}
|S_{[\zeta k]}(q^{-1}f)(\Phi_{\a,\beta}^m(x,y))-S_{[\zeta k]}(q^{-1}f)(h_m)|\leq q^{-1}\zeta^{3/2}C' \delta_x k^{3/2}.
\end{equation}
By Taylor formula (and the chain rule), for some $\theta_k\in \T^2$,
\begin{multline}\label{eq:sply}
S_{[\zeta k]}(q^{-1}f)(\Phi_{\a,\beta}^m(x',y'))-S_{[\zeta k]}(q^{-1}f)(h_m) \\ =(m(x-x')+(y-y'))S_{[\zeta k]}(q^{-1}\frac{\partial f}{\partial y})(\Phi_{\a,\beta}^m(x',y'))\\ +
(m(x-x')+(y-y'))^2 S_{[\zeta k]}(q^{-1}\frac{\partial^2 f}{\partial y^2})(\theta_k).
\end{multline}
For $m\leq T^{1+1/10}$, we have $\vert m(x-x')+(y-y')\vert \leq m\delta_x +\delta_y \leq T^{-1/3}$ and therefore, by Lemma \ref{lem:upperbound} for $g=\frac{\partial^2 f}{\partial y^2}$,
$$
|(m(x-x')+(y-y'))^2
S_k(\frac{\partial^2 f}{\partial y^2})(\theta_k)|\leq T^{-1/10}.
$$
Let now $R''\gg R'$ (to be specified later) and let $m\in [R''\max(k,T ),R''\max(k,T)+D_\eta \zeta k]$ (see Lemma \ref{lem:lowbound}) be such that 
$$
S_{[\zeta k]}(q^{-1}\frac{\partial f}{\partial y})(\Phi_{\a,\beta}^m(x',y'))\geq c'q^{-1}\zeta^{1/2} k^{1/2}.
$$
Then since $R'\delta_x^{1/3}\geq \delta_y$ and we are in  \textbf{Case 2}, for $R''>1$ large enough we have the inequalities
$m\delta_x \geq R'' T \delta_x \geq  10 \delta_y$, hence by \eqref{eq:sply} we have
\begin{equation}\label{eq:00}
|S_{[\zeta k]}(q^{-1}f)(\Phi_{\a,\beta}^m(x',y'))-S_{[\zeta k]}(q^{-1}f)(h_m)|\geq \frac{c'}{2}q^{-1}\zeta^{1/2}k^{1/2}m\delta_x.
\end{equation}
Since $R'\delta_x^{1/3} \geq \delta_y$, $m\geq R'' \max(k,T)$ and we are in \textbf{Case 2}, for $R''>1$ large enough we have
the following lower bound:
$$
\frac{c'}{2}q^{-1}\zeta^{1/2}k^{1/2}m\delta_x\geq 10\max(C'p^{-1}k^{1/2}\delta_w,q^{-1}\zeta^{3/2}C'k^{3/2}\delta_x)
$$
and, for any $k\geq T$, we have $\frac{c'}{2}q^{-1}\zeta^{1/2}k^{1/2}m\delta_x\geq d'''>0$.
This, together with \eqref{sar:split2}, \eqref{eq:00}, \eqref{eq:01} and \eqref{eq:02}, implies that 
\begin{multline*}
|\left(S_k(p^{-1}f)(\Phi_{\a,\beta}^m(z,w))-S_k(p^{-1}f)(\Phi_{\a,\beta}^m(z,w'))\right)-\\
\left(S_{[\zeta k]}(q^{-1}f)(\Phi_{\a,\beta}^m(x,y))-S_{[\zeta k]}(q^{-1}f)(\Phi_{\a,\beta}^m(x',y'))\right)|\geq d'''/2.
\end{multline*}
Thus, by cocycle identity, \eqref{eq:nwc} follows for $s=k$ or $s=k+m$.
This finishes the proof.
\end{proof}

We can now prove Theorem \ref{thm:main2}.

\begin{proof}[Proof of Theorem \ref{thm:main2}] We will prove that the hypotheses of the Disjointness Criterion for special flows given in 
Proposition~\ref{cocycle}  are satisfied. Theorem \ref{thm:main2} will then follow.

Let $A_k:(\T^2,\lambda_{\T^2})\to (\T^2,\lambda_{\T^2})$ be given by $A_k(x,y)=(x,y+\frac{1}{k})$. Obviously $A_k\to Id$ uniformly. 
Fix $\epsilon>0$ and $N\in \N$. Let us consider any 
$(x,y),(x',y')\in \T^2$ with $d((x,y),(x',y'))<\delta$ and any $(z,w),(z,w')=A_k(z,w)$ with $k<\delta^{-1}$. Let 
$$
 T:=\min(\delta_x^{-2/3},\delta_y^{-2},\delta_w^{-2}).
$$
Let us define $a_s:=a_{1,s}-a_{2,s}$ with 
$$
\begin{aligned} 
&a_{1,s}:=S_s(p^{-1}f)(z,w)-S_s(p^{-1}f)(z,w')\,, \\ & a_{2,s}:=S_{[\zeta s]}(q^{-1}f)(x,y)-S_{[\zeta s]}(q^{-1}f)(x',y')\,.
\end{aligned} 
$$

Let $D'''>0$ be as in Proposition \ref{prop:cruc}. For every $s\in [0,2D'''T]$, we have 
\begin{multline}\label{eq:clos}
d_2((\Phi_{\a,\beta}^s(z,w)),(\Phi_{\a,\beta}^s(z,w')))=\delta_w, \, 
d_1((\Phi_{\a,\beta}^s(x,y)),(\Phi_{\a,\beta}^s(x',y')))=
\delta_x\text{ and }\\
d_2((\Phi_{\a,\beta}^s(x,y)),(\Phi_{\a,\beta}^s(x',y')))\leq s\delta_x+\delta_y\leq 2D'''\delta_x^{1/3}+\delta_y, 
\end{multline}

Hence (since $f$ is $C^1$), for every $\kappa>0$ there exists $\delta_\kappa>0$ such that, for $\delta\in (0,\delta_\kappa)$ and
for every $s\in [0,2D'''T]$, we have 
$$|a_{s+1}-a_s|<\kappa.$$ 
Thus by Proposition \ref{prop:cruc} there exists $M\in [0,D'''T]$ such that 
$$
\min\{ |a_M+d'''|, |a_M- d''' \} \leq \epsilon/3.
$$
This gives \eqref{forw.coc}.
Notice that, for any given $N\in \N$, there exists $\delta_N>0$ such that, for $\delta \in (0,\delta_N)$ we have 
$M\geq N$, and by Lemma \ref{lem:diver} (for $a_{1,M},a_{2,M}$), by the definition of $T>0$, 
there exists $V>0$, independent of $\epsilon>0$, $\delta>0$ and $N\in N$, such that
$$
|a_M|\leq V,
$$
for some $V>0$. This gives \eqref{eq.bounded}.

\smallskip
Let $L=\kappa M$. Notice that $[0,(1+\zeta)(M+L)]\subset [0,2D'''T]$ and therefore by \eqref{eq:clos}
\begin{equation}
d(\Phi_{\a,\beta}^n(x,y),\Phi_{\a,\beta}^n(x',y'))<\epsilon, \text{ for every } n\in [0,\max(1+\zeta)M+L]
\end{equation}
and analogously for $(z,w), (z,w')$. This gives \eqref{base:close}.

\smallskip

By the cocycle identity and by Lemma \ref{lem:diver}, there exists $\kappa'''_\epsilon>0$ such that for $\kappa\in (0, \kappa'''_\epsilon)$,
for all $u,w\in [0,2\zeta(M+L)]$ with $|u-w|\leq \kappa M$, we have
\begin{multline*}
|a_{1,w}-a_{1,u}|=|S_{w-u}(p^{-1}f)(\Phi_{\a,\beta}^u(z,w))-S_{w-u}(p^{-1}f)(\Phi_{\a,\beta}^u(z,w'))|\\ \leq C''|u-w|^{1/2}\delta_w\leq 
C''\kappa^{1/2}D^{1/2}T^{1/2}\delta_w\leq \epsilon\,.
\end{multline*}

Analogously we show that $|a_{2,u}-a_{2,w}|\leq \epsilon/10$:\\
$$
|a_{2,w}-a_{2,u}|=|S_{\zeta_{w,u}}(q^{-1}f)(\Phi_{\a,\beta}^{[\zeta u]}(x,y))-
S_{\zeta_{w,u}}(q^{-1}f)(\Phi_{\a,\beta}^{[\zeta u]}(x',y'))|, 
$$
for some $\zeta_{w,u}\leq \zeta C_0|w-u|$. By Lemma \ref{lem:diver} (for $g=f-\int_{\T^2}f d\lambda_{\T^2}$),  \eqref{eq:clos} for $s=[\zeta M]$, 
there exist $C''>0$ and $\kappa'_\epsilon>0$ such that for $\kappa\in (0, \kappa''_\epsilon)$, 

we have (since $|w-u|\leq L=\kappa M\leq D'''\kappa T$) (for some $C'''>0$)
\begin{multline*}
|S_{\zeta_{w,u}}(q^{-1}f)(\Phi_{\a,\beta}^{[\zeta u]}(x,y))-
S_{\zeta_{w,u}}(q^{-1}f)(\Phi_{\a,\beta}^{[\zeta u]}(x,y'))|\leq\\ C'''(D'''\kappa T)^{3/2}\delta_x+ C'''(D'''\kappa T)^{1/2}(2D'''\delta_x^{1/3}+\delta_y)\leq\\
 C'''(D'''\kappa)^{3/2}+\kappa^{1/2}C'''D'''(2D'''+1)\leq \epsilon/10.
\end{multline*}
\smallskip This finishes the proof of \eqref{forw.tight} and hence by Proposition~\ref{cocycle} completes the proof of Theorem \ref{thm:main2}. 
\end{proof}

\end{document}